\documentclass[a4paper, 10pt]{article}

\usepackage[english]{babel}	
\usepackage{lmodern}
\usepackage[T1]{fontenc}
\usepackage[utf8]{inputenc}	

\usepackage{{a4wide}}
\usepackage{amsfonts, amsmath, amssymb,amsthm} 
\usepackage{mathtools}
\usepackage{mathrsfs}	
\usepackage{bbm}	
\DeclareMathAlphabet{\mathbbold}{U}{bbold}{m}{n}	
\usepackage{soul}

\usepackage{comment}
\usepackage{hyperref}

\usepackage[strings]{underscore}

\usepackage{enumitem}
\usepackage{array}
\usepackage[dvipsnames]{xcolor}
\usepackage{datetime}
\usepackage{cases}
\usepackage{upgreek}

\usepackage[all]{xy} 

\newcommand{\R}{\mathbb{R}}
\newcommand{\Z}{\mathbb{Z}}
\newcommand{\C}{\mathbb{C}}
\newcommand{\N}{\mathbb{N}}

\renewcommand{\mid}{\;:\;}

\renewcommand{\ge}{\geqslant}
\renewcommand{\leq}{\leqslant}
\renewcommand{\geq}{\geqslant}

\newcommand{\eps}{\varepsilon}
\renewcommand{\epsilon}{\varepsilon}

\newcommand{\lang}{\langle}
\newcommand{\rang}{\rangle}

\newcommand{\T}[1]{{1}^\mathrm{T}}

\newcommand{\ie}{\emph{i.e.,~}}
\newcommand{\eg}{\emph{e.g.,~}}

\newcommand{\Id}{\mathbbm{I}}

\newcommand*\diff{\mathop{}\!\mathrm{d}}

\newcommand{\cv}{\to}
\newcommand{\cvl}[1]{\underset{#1}{\longrightarrow}}
\newcommand{\cvf}{\overset{w}{\rightharpoonup}}

\newcommand{\tow}{\cvf}

\newcommand{\GG}{\mathbb{G}}

\numberwithin{equation}{section}

\DeclareMathOperator{\argmin}{argmin}

\newtheorem{thrm}{Theorem}[section]
\newtheorem{crllr}[thrm]{Corollary}
\newtheorem{lmm}[thrm]{Lemma}
\newtheorem{prpstn}[thrm]{Proposition}

\theoremstyle{definition}
\newtheorem{dfntn}[thrm]{Definition}
\newtheorem{assumption}[thrm]{Assumption}
\newtheorem{condition}[thrm]{Condition}
\newtheorem{xmpl}[thrm]{Example}
\newtheorem{rmrk}[thrm]{Remark}

\newcommand{\Kx}{\mathcal{K}}
\newcommand{\Kw}{\widehat{\mathcal{K}}}
\newcommand{\Ke}{\left( \Kw - \plong(\Kx) \right)}
\newcommand{\Kz}{\Xi}

\newcommand{\Kwn}{\mathcal{K}}
\newcommand{\wdyn}{w}
\newcommand{\rot}{A}
\newcommand{\abs}[1]{\left\vert#1\right\vert}

\newcommand{\one}{\mathbbold{1}}
\newcommand{\NN}{\mathcal{N}}

\newcommand{\fonction}[5]{
\begin{equation*}
\displaystyle
\begin{array}{lrcl}
{#1}: & #2 & \longrightarrow & #3 \\
    & #4 & \longmapsto & #5
\end{array}
\end{equation*}}

\newcommand{\etat}{z}
\newcommand{\etath}{\hat{\etat}}
\newcommand{\zhat}{\etath}
\newcommand{\mA}{\mathcal{A}}
\newcommand{\mB}{\mathcal{B}}
\newcommand{\mC}{\mathcal{C}}
\newcommand{\mQ}{\mathcal{Q}}

\newcommand{\D}{\mathcal{D}}
\newcommand{\Dd}{\mathcal{D}_\delta}

\newcommand{\e}{e}

\newcommand{\re}{r_\eps}

\newcommand{\XX}{Z}
\newcommand{\YY}{Y}
\newcommand{\Kclass}{\mathcal{K}}
\newcommand{\Kinf}{\Kclass_\infty}

\newcommand{\lin}{\mathscr{L}}
\newcommand{\linx}{\lin(\XX)}
\newcommand{\Ux}{\mathcal{U}_x}
\newcommand{\Uw}{\mathcal{U}_\wdyn}
\newcommand{\Xd}{\mathcal{X}_0}

\newcommand{\xhat}{\hat{x}}
\newcommand{\psn}{\mathcal{N}}
\newcommand{\dom}{\mathcal{D}}
\newcommand{\doma}{\dom}
\newcommand{\sg}{\mathbb{T}}
\newcommand{\sgeps}{\mathbb{S}}
\newcommand{\evol}{\sg}
\newcommand{\evoleps}{\sgeps}
\newcommand{\opa}{\mA}
\newcommand{\opc}{\mC}
\newcommand{\rr}{\alpha}
\newcommand{\mes}{y}

\newcommand{\plong}{\uptau}
\newcommand{\inv}{\uppi}
\newcommand{\qform}{Q}
\newcommand{\norm}[1]{\left\|#1\right\|_\XX}
\newcommand{\normA}[2]{\left\|#1\right\|_{#2}}
\newcommand{\psX}[2]{\lang#1,#2\rang_\XX}

\newcommand{\xdot}{\dot{x}}
\newcommand{\matR}{\rot}
\newcommand{\RR}{\mathcal{R}}
\renewcommand{\S}{S}
\newcommand{\rep}{\rho}
\newcommand{\ray}{r}
\newcommand{\B}{B}
\newcommand{\mm}{\mathfrak{m}}
\newcommand{\linxy}{\lin(\XX, \C^\mm)}
\newcommand{\hh}{\mathfrak{h}}
\newcommand{\yy}{\mathfrak{y}}
\newcommand{\Ra}{R_0}
\newcommand{\Rb}{R_1}
\newcommand{\Rc}{R_2}
\newcommand{\dt}{\Delta}
\newcommand{\oo}{o}

\newcommand{\bornuval}{\kappa \frac{\jj}{\mu} + 16 \nu^2 \delta }
\newcommand{\bornu}{u_{\max}}
\newcommand{\psnlim}{\psn^2_\infty}
\newcommand{\carac}{\mathcal{I}}
\newcommand{\TT}{\mathcal{T}}

\newcommand{\jj}{j}

\newcommand{\angleS}{\theta}
\newcommand{\varS}{s}
\newcommand{\kmin}{k_{\min}}
\newcommand{\sousz}{I}
\newcommand{\ttau}{t'}
\newcommand{\frakr}{\mathfrak{R}}
\newcommand{\La}{\mathcal{L}_\alpha}
\newcommand{\Lf}{\ell_{\invf}}
\newcommand{\Linv}{\ell_{\inv}}
\newcommand{\Lplong}{\ell_{\plong}}
\newcommand{\invf}{\mathfrak{f}}

\usepackage{accents}
\newcommand{\ubar}[1]{\underline{#1}} %

\newcommand{\BB}{\mathcal{B}}
\newcommand{\EE}{\mathcal{E}}

\title{New perspectives
on output feedback stabilization\\
at an unobservable target\footnote{This research was partially funded by the French Grant ANR ODISSE (ANR-19-CE48-0004-01)
and by the ANR SRGI (ANR-15-CE40-0018).}}

\usepackage{authblk}

\author[1]{Lucas Brivadis}
\affil[1]{
Univ. Lyon, Universit\'e Claude Bernard Lyon 1, CNRS, LAGEPP UMR 5007, 43 bd du 11 novembre 1918, F-69100 Villeurbanne, France,
\texttt{lucas.brivadis@univ-lyon1.fr}, \texttt{ulysse.serres@univ-lyon1.fr}.}

\author[2]{Jean-Paul Gauthier}
\affil[2]{Universit\'e de Toulon, Aix Marseille Univ, CNRS, LIS, France, \texttt{jean-paul.gauthier@univ-tln.fr}.} 
\author[3]{Ludovic Sacchelli}
\affil[3]{
Inria, Université Côte d’Azur, LJAD, CNRS, MCTAO team, Sophia Antipolis, France, \texttt{ludovic.sacchelli@inria.fr}.}

\author[1]{Ulysse Serres}

\date{\today}

\begin{document}

\maketitle

\begin{abstract}
We address the problem of dynamic output feedback stabilization at an unobservable target point.
The challenge lies in according the antagonistic nature of the objective and the properties of the system:
the system tends to be less observable as it approaches the target.
We illustrate two main ideas: well chosen perturbations of a state feedback law can yield new observability properties of the closed-loop system, and embedding systems into bilinear systems admitting observers with dissipative error systems allows to mitigate the observability issues.
We apply them on a case of systems with linear dynamics and nonlinear observation map and make use of an \emph{ad hoc} finite-dimensional embedding.
More generally, we introduce a new strategy based on infinite-dimensional unitary embeddings.
To do so, we extend the usual definition of dynamic output feedback stabilization in order to allow infinite-dimensional observers fed by the output.
We show how this technique, based on representation theory, may be applied to achieve output feedback stabilization at an unobservable target.
\end{abstract}


\section{Introduction}

The problem of output feedback stabilization is one of deep interest in control theory.
It consists of stabilizing the state of a dynamical system, that is only partially known, to a target point.
Although a vast literature tackles this topic (see \cite{AndrieuPraly2009}, and references therein), some fundamental problems remain mostly open.
The issue can be formulated in the following manner.
Let $n$, $m$ and $p$ be positive integers, $f:\R^n\times\R^p\to\R^n$ and $h:\R^n\to\R^m$.
For all $u\in C^0(\R_+, \R^p)$, consider the following observation-control system:
\begin{equation}\label{E:system_general}
\left\{
\begin{aligned}
&\dot{x}=f(x, u)
\\
&y= h(x)
\end{aligned}
\right.
\end{equation}
where $x$ is the state of the system, $u$ is the control (or input) and $y$ is the observation (or output).
We assume that $f$ is uniformly locally Lipschitz with respect to $x$ and continuous.
We may as well assume that the target point at which we aim to stabilize $x$ is the origin $0\in\R^n$, $h(0)=0$ and $f(0, 0)=0$.

\begin{dfntn}[Dynamic output feedback stabilizability]\label{def:stab_out}
System~\eqref{E:system_general} is said to be \emph{locally} (resp. \emph{globally}) \emph{stabilizable by means of a dynamic output feedback} if and only if the following holds.

There exist two continuous map $\nu:\R^q\times\R^p\times\R^m\to\R^q$ and $\varpi:\R^q\times\R^m\to\R^p$ for some positive integer $q$ such that $(0, 0)\in\R^n\times\R^q$ is a locally (resp. globally) asymptotically stable\footnote{
Recall that a dynamical system is said to be asymptotically stable at an equilibrium point with some basin of attraction if and only if
each initial condition in the basin of attraction yields at least one solution to the corresponding Cauchy problem, each solution converges to the equilibrium point, and the equilibrium point is Lyapunov stable.
}
equilibrium point of the following system:
\begin{equation}\label{E:system_stab}
\left\{
\begin{aligned}
&\dot{x}=f(x, u)
\\
&y= h(x)
\end{aligned}
\right.
,\qquad
\left\{
\begin{aligned}
&\dot{\wdyn}=\nu(\wdyn, u, y)
\\
&u= \varpi(\wdyn, y).
\end{aligned}
\right.
\end{equation}

Additionally, if for any compact set $\Kx\subset\R^n$, there exist two continuous maps $\nu:\R^q\times\R^p\times\R^m\to\R^q$ and $\varpi:\R^q\times\R^m\to\R^p$ for some positive integer $q$, and a compact set $\Kw\subset\R^q$ such that $(0, 0)\in\R^n\times\R^q$ is an asymptotically stable equilibrium point of \eqref{E:system_stab} with basin of attraction containing $\Kx\times\Kw$,
then \eqref{E:system_general} is said to be \emph{semi-globally stabilizable by means of a dynamic output feedback}.
\end{dfntn}

Without loss of generality, we assume that if \eqref{E:system_stab} is locally asymptotically stable at $(0, 0)$, then the value of the control at the target point is zero: $\varpi(0, 0)=0\in\R^p$.

A common strategy to achieve dynamic output feedback stabilization consists in finding a stabilizing state feedback, designing an observer system that learns the state from the dynamics of its output, and using as an input the state feedback applied to the observer.
For linear systems, this corresponds to the so-called separation principle, which consists of designing ``separately'' a stabilizing state feedback law and a state observer. 
This strategy is known to fail in general for nonlinear systems.
In \cite{jouan} and  \cite{TeelPraly1994, TeelPraly1995}, the authors proved under a \emph{complete uniform observability} assumption that any system semi-globally stabilizable by means of a static state feedback is also semi-globally stabilizable by means of a dynamic output feedback.

\begin{dfntn}[Observability]\label{def:obs}
System~\eqref{E:system_general} is observable for some input $u\in C^0(\R_+, \R^p)$ in time $T>0$ if and only if,
any two solutions $x, \tilde{x}$ of \eqref{E:system_general} whose corresponding outputs $y$ and $\tilde{y}$ are equal almost everywhere on $[0, T]$, must also be equal almost everywhere on $[0, T]$.
\end{dfntn}
Complete uniform observability required in \cite{TeelPraly1994, TeelPraly1995}
implies
observability for all inputs and all times.
However, as proved in \cite{Gauthier_book}, 
it is not generic for nonlinear systems to be completely uniformly observable when 
the dimension of the output is less than or equal to the dimension of the input.
The problem of dynamic output feedback stabilization remains open when singular inputs (that are, inputs that make the system unobservable) exist.
Crucially, observability singularities prevent from applying classical tried-and-tested methods. However, such difficulties occur in practical engineering systems,
where original strategies need to be explored
\cite{pasillasxbs,Rapaport,9172770,ludo,flaya2019,surroop2020adding,combes2016adding}, leading to a renewal of interest in the issue in recent years.
The main difficulties arise when the control input
corresponding to the equilibrium point of \eqref{E:system_stab} is singular.
Indeed, a contradiction may occur between the stabilization of the state at the target point, and the fact that the observer system may fail to properly estimate the state near the target.
Various techniques have been introduced to remove this inconsistency, on which this paper focuses.

Most of them rely on a
modification
of the input, that helps the observer system to estimate the state even near the target point.
This strategy was employed in \cite{Coron1994} (that achieved local stabilization by using a periodic time-dependent perturbation), in \cite{ShimTeel2002} (that achieved practical stabilization by using a ``sample and hold'' time-dependent perturbation),
and more recently in \cite{brivadis:hal-03180479} (that achieved global stabilization on a specific class of systems).
Let us also mention the works of \cite{combes2016adding, flaya2019, surroop2020adding}, that also rely on a high-frequency excitation of the input.
Adopting another point of view in line with \cite{MarcAurele, brivadis2019avoiding}, we are interested in
 time-independent
perturbations of the feedback laws.

Another important tool in stabilization theory
is the use of weakly contractive systems, for which flows are non-expanding \cite{andrieu:hal-02611605,Manchester_contraction2014,Jafarpour_contraction2020,jq,mazenc}.
The fact that the error between the state and the observer
has contractive dynamics has proven to be a powerful tool in \cite{MarcAurele, sacchelli2019dynamic},
because it guarantees that the error is non-increasing, no matter the observability properties of the system.
More precisely,
in \cite{MarcAurele} a strategy of feedback perturbation is used in conjunction with the contraction property of a quantum control system to achieve stabilization at an unobservable target.
In \cite{sacchelli2019dynamic}, the authors proved that for non-uniformly observable state-affine systems, detectability at the target is a sufficient condition to set up a separation principle.
In order to access such contraction properties on a wider class of systems, we turn to embeddings techniques. Embedding the original nonlinear system into a bilinear system and design an observer with dissipative error is the second guideline that we aim to follow.

In the present paper, we wish to coalesce the insights provided by these previous works in order to come up with solutions to attack the problem of dynamic output feedback stabilization at an unobservable target.
Embedding techniques have to rely on the systems structure. Inspired by an example from \cite{Coron1994}, we focus here on systems with linear conservative dynamics and nonlinear observation maps. The strategies developed in the paper match the restrictions we impose on the systems, but aim to illustrate a more general framework in which to attack stabilization at unobservable targets. 

A direct method to linearize the output is to consider it as an additional state coordinate. If an observer with dissipative error system can be found for the new embedded system, 
we prove that a feedback law perturbation approach can be efficient for dynamic output feedback stabilization.
However, the existence of such an observer is not guaranteed.
We attempt to overcome this difficulty by proposing embeddings into infinite-dimensional unitary systems with linear output. As illustrated in \cite{Celle-etal.1989}, this idea helps in designing dissipative observers for a wider array of nonlinear systems. Furthermore, infinite-dimensional observers are not limited by topological obstructions that may occur in the finite-dimensional context (see Section~\ref{sec:coron}). 
Following this approach in the context of output feedback stabilization
yields a coupled ODE--PDE system that demands an \emph{ad-hoc} functional framework. We set up this strategy on a two-dimensional system presenting an archetypal singularity at the target point.

\bigskip

\noindent
\textbf{Content.} Necessary conditions for dynamic output feedback stabilization are discussed in Section~\ref{sec:nc}.
The finite-dimensional strategy is explored in Section~\ref{sec:finite} (first, topological obstructions are raised in \ref{sec:coron}, while a positive stabilization result is proved in \ref{sec:converse}). 
In Section~\ref{sec:infinite},
we set up preliminaries for our infinite-dimensional strategy. In the final Section~\ref{sec:example_infinie}, 
these concepts are applied to achieve a stabilization result on a coupled ODE--PDE state-observer system.

\bigskip

\noindent
\textbf{Notations.}
Denote by $\R$ (resp. $\R_+$) the set of real (resp. non-negative) numbers, by $\C$ the set of complex numbers and by $\Z$ (resp. $\N$) the set of integers (resp. non-negative integers).
The Euclidean norm over $\R^n$ (or $\C^n)$ is denoted by $|\cdot|$ for any $n\in\N$.
The real and imaginary parts of $z\in\C$ are denoted by $\Re z$ and $\Im z$, respectively.
For any matrix $A\in\R^{n\times m}$, the transpose of $A$ is denoted by $A'$.
For any normed vector space
$(\XX, \|\cdot\|_\XX)$, we denote by $B_\XX(x, r)$ (resp. $\bar{B}_\XX(x, r)$) the open (resp. closed) ball of $\XX$ centered at $x\in \XX$ of radius $r>0$ for the norm $\|\cdot\|_\XX$.
The identity operator over $\XX$ is denoted by $\Id_\XX$.
For all $k\in\N\cup\{\infty\}$ and all interval $U\subset\R$, the set $C^k(U, \XX)$ is the set of $k$-continuously differentiable functions from $U$ to $\XX$.

\medskip

For any Hilbert space $\XX$, denote by $\psX{\cdot}{\cdot}$ the inner product over $\XX$ and $\norm{\cdot}$ the induced norm.
If $\YY$ is also a Hilbert space, then $\lin(\XX, \YY)$ denotes the space of bounded linear maps from $\XX$ to $\YY$, and $\linx = \lin(\XX, \XX)$.
We identify the Hilbert spaces with their dual spaces via the canonical isometry, so that the adjoint of $\opc\in\lin(\XX, \YY)$, denoted by $\opc^*$, lies in $\lin(\YY, \XX)$.
Let us recall the characterization of the strong and weak topologies on $\XX$.
A sequence $(x_n)_{n\geq0}\in\XX^\N$ is said to be strongly convergent to some $x^\star\in\XX$ if
$\norm{x_n-x^\star}\cv 0$ as $n\cv+\infty$, and we shall write $x_n\cv x^\star$ as $n\cv+\infty$.
It is said to be weakly convergent to $x^\star$ if
$\psX{x_n-x^\star}{\psi}\cv 0$ as $n\cv+\infty$ for all $\psi\in\XX$, and we shall write $x_n\cvf x^\star$ as $n\cv+\infty$.
The strong topology on $\XX$ is finer than the weak topology (see, \eg \cite{Brezis} for more properties on these usual topologies).

\section{Necessary conditions}\label{sec:nc}

The aim of the paper is to discuss dynamic output feedback stabilization strategies in the presence of observability singularities in contrast to 
\cite{TeelPraly1994, TeelPraly1995}, where uniform observability is assumed. However, in trying to weaken the observability assumption in the context dynamic output feedback stabilization, one first needs to check that the goal is still achievable. In this short section, necessary conditions for dynamic output feedback stabilizability are discussed. These should be put in  perspective with sufficient conditions that can be found in the literature, as well as those we exhibit in the paper.

\begin{dfntn}[State feedback stabilizability]\label{def:stab}
System~\eqref{E:system_general} is said to be \emph{locally} (resp. \emph{globally}) \emph{stabilizable by means of a (static) state feedback} if and only if
there exists a continuous map $\phi:\R^n\to\R^p$ such that
$0\in\R^n$ is a locally (resp. globally) asymptotically stable equilibrium point of
\begin{equation}\label{E:system_stab_state}
\left\{
\begin{aligned}
&\dot{x}=f(x, u)
\\
&u=\phi(x).
\end{aligned}
\right.
\end{equation}

Additionally, if for any compact set $\Kx\subset\R^n$,
there exists a continuous map $\phi:\R^n\to\R^p$ such that
$0\in\R^n$ is an asymptotically stable equilibrium point of \eqref{E:system_stab_state} with basin of attraction containing $\Kx$,
then \eqref{E:system_general} is said to be \emph{semi-globally stabilizable by means of a static state feedback}.

\end{dfntn}

The problem of dynamic state feedback stabilization of \eqref{E:system_general} is equivalent to the dynamic output feedback stabilization in the case where $h(x)=x$.
Therefore, dynamic state feedback stabilizability of \eqref{E:system_general} is a necessary condition for dynamic output feedback stabilizability.
One may wonder if \emph{static} state feedback stabilizability of \eqref{E:system_general} is  necessary  for dynamic output feedback stabilizability.
In \cite{AndrieuPraly2009}, the authors answer by the positive if a sufficiently regular selection function can be found. We recall their result below.

\begin{thrm}[\!\!{\cite[Lemma 1, (1)]{AndrieuPraly2009}}]
Assume that \eqref{E:system_stab} is locally asymptotically stable at $(0, 0)$ with basin of attraction\footnote{In \cite[Lemma 1, (1)]{AndrieuPraly2009},
the authors state only a global version of the result, that is, $\Ux=\R^n$ and $\Uw=\R^q$. However, the proof remains identical in the other cases.
}
$\Ux\times\Uw$. Let $V$ be a $C^\infty(\Ux\times\Uw, \R_+)$ strict proper Lyapunov function of \eqref{E:system_stab}.
If there exists a selection map $\Ux\ni x\mapsto \phi(x)\in \argmin_{\Uw}V(x,\cdot)$ which is locally Hölder of order strictly larger than $\frac{1}{2}$, then
\eqref{E:system_stab_state} is locally asymptotically stable at $0$ with basin of attraction containing $\Ux$.
\end{thrm}

Therefore, up to the existence of a sufficiently regular selection map, this result implies that the following local (resp. semi-global, global) condition is necessary for the local (resp. semi-global, global) stabilizability of \eqref{E:system_general} by means of a dynamic output feedback.

\begin{condition}[State feedback stabilizability
---
local, semi-global, global]\label{hyp:state_feedback}
System~\eqref{E:system_general} is locally (resp. semi-globally, globally) stabilizable by means of a static state feedback.
\end{condition}

In \cite{Coron1994}, J.-M. Coron stated two additional conditions that he proved to be sufficient when local static state feedback stabilizability holds to ensure local dynamic output feedback stabilizability, provided that one allows the output feedback to depend on time (which we do \emph{not} allow in this paper).
The two following conditions are weaker versions of the ones of \cite{Coron1994}. We prove that these two conditions are necessary to ensure dynamic output feedback stabilizability.
The first one, known as \emph{$0$-detectability} is also used by E. Sontag in \cite{sontag1981conditions} in the context of abstract nonlinear regulation theory.

Before stating this condition, let us recall the following.
For any input $u\in C^0(\R_+, \R^n)$,
and any initial condition $x_0\in\R^n$,
there exists exactly one maximal solution of \eqref{E:system_general}, defined on $[0, T(x_0, u))$. This solution, denoted by $\varphi_t(x_0, u)$, is such that $\varphi_0(x_0, u) = x_0$ and $\frac{\partial\varphi_t(x_0, u)}{\partial t} = f( \varphi_t(x_0, u), u(t))$ for all $t\in[0, T(x_0, u))$.

\begin{condition}[0-detectability
---
local, global]\label{hyp:distinguish}
Let $\Xd = \{x_0\in\R^n\mid \forall t\in[0, T(x_0, 0)),$ $h(\varphi_t(x_0, 0))=0\}$.
Then $0\in\Xd$ is a locally (resp. globally) asymptotically stable equilibrium point of the vector field $\Xd\ni x\mapsto f(x, 0)$.
\end{condition}

\begin{thrm}\label{th:necessary_distinguish}
If \eqref{E:system_general} is locally (resp. semi-globally, globally) stabilizable by means of a dynamic output feedback, then  Condition~\ref{hyp:distinguish} holds locally (resp. globally, globally).
\end{thrm}
\begin{proof}

The set $\Xd$ is invariant for the vector field $x\mapsto f(x, 0)$ and $0\in\Xd$.
Let $x_0\in\Xd$.
Assume that \eqref{E:system_general} is locally stabilizable by means of a dynamic output feedback, and that $(x_0, 0)$ is in the basin of attraction of $(0, 0)$ for \eqref{E:system_stab}.

Then $t\mapsto (\varphi_t(x_0, 0), 0)$ is a trajectory of \eqref{E:system_stab} with initial condition $(x_0, 0)$.
Hence $\varphi_t(x_0, 0)$ is well-defined for all $t\geq0$ and tends towards $0$ as $t$ goes to infinity. Moreover, for all $R>0$, there exists $r>0$ such that, if $x_0\in B_{\R^n}(x_0, r)$, then $\varphi_t(x_0, 0)\in B_{\R^n}(x_0, R)$ for all $t\geq0$.

If we assume that \eqref{E:system_general} is globally stabilizable by means of a dynamic output feedback,
then the arguments still hold for any
$x_0\in\R^n$.
If \eqref{E:system_general} is only semi-globally stabilizable by means of a dynamic output feedback,
we first define $\Kx$ as in Definition~\ref{def:stab_out} containing $x_0$.
\end{proof}

\begin{condition}[Indistinguishability $\Rightarrow$ common stabilizability
---
local, global]\label{hyp:common}
For all $x_0$, $\tilde{x}_0$ in some neighborhood of $0\in\R^n$ (resp. for all $x_0$, $\tilde{x}_0$ in $\R^n$),
if for all $u\in C^0(\R_+, \R^p)$ such that
$T(x_0, u) = +\infty$
it holds that
$h(\varphi_t(x_0, u)) = h(\varphi_t(\tilde{x}_0, u))$
for all  $t\in[0, T(\tilde{x}_0, u))$,
then there exists $v\in C^0(\R_+, \R^p)$ such that
$\varphi_t(x_0, v)$ and $\varphi_t(\tilde{x}_0, v)$
are well-defined for all $t\in\R_+$ and
tend towards $0$ as $t$ goes to infinity.
\end{condition}

\begin{thrm}\label{th:necessary_common}
If \eqref{E:system_general} is locally (resp. semi-globally, globally) stabilizable by means of a dynamic output feedback, then  Condition~\ref{hyp:common} holds locally (resp. globally, globally).
\end{thrm}
\begin{proof}
Let $x_0, \tilde{x}_0\in \R^n$ be such that for all $u\in C^0(\R_+, \R^p)$ such that
$T(x_0, u) = +\infty$
it holds that
$h(\varphi_t(x_0, u)) = h(\varphi_t(\tilde{x}_0, u))$
for all  $t\in[0, T(\tilde{x}_0, u))$,
Assume that \eqref{E:system_general} is locally stabilizable by means of a dynamic output feedback, and that $(x_0, 0)$, $(\tilde{x}_0, 0)$ are in the basin of attraction of $(0, 0)$ for \eqref{E:system_stab}.

Let $(x, \wdyn)$ be a solution of \eqref{E:system_stab} starting from $(x_0, 0)$.
Set $v = \varpi(\wdyn, h(x))$.
Then $T(x_0, v) = +\infty$ and $\varphi_t(x_0, v) \to 0$ as $t\to+\infty$.
Let $\tilde{x}(t) = \varphi_t(\tilde{x}_0, v)$ for all $t\in[0, T(\tilde{x}_0, v))$.
Since $h(\varphi_t(x_0, v)) = h(\varphi_t(\tilde{x}_0, v))$ for all $t\in[0, T(\tilde{x}_0, v))$,
$(\tilde{x}, \wdyn)$ is a solution of
\eqref{E:system_stab} starting from $(\tilde{x}_0, 0)$.
Hence $T(\tilde{x}_0, v)=+\infty$ and $\varphi_t(\tilde{x}_0, v) \to 0$ as $t\to+\infty$.

If we assume that \eqref{E:system_general} is globally stabilizable by means of a dynamic output feedback,
then the arguments still hold for any $x_0, \tilde{x}_0\in\R^n$.
If \eqref{E:system_general} is only semi-globally stabilizable by means of a dynamic output feedback,
we first define $\Kx$ as in Definition~\ref{def:stab_out} containing $x_0$ and $\tilde{x}_0$.
\end{proof}

\begin{rmrk}
In \cite{sacchelli2019dynamic}, the authors consider the problem of dynamic output feedback stabilization of \emph{dissipative} state-affine systems, that is, systems of the form
\begin{equation}\label{E:diss}
\left\{
\begin{aligned}
&\dot{x}= A(u)x+B(u)
\\
&y= C x
\end{aligned}
\right.
\end{equation}
where there exists some positive definite matrix $P\in\R^n\times\R^n$ such that, for all inputs $u$ in some admissible set,
\begin{align}\label{eq:diss}
    PA(u)+A(u)'P\leq0.
\end{align}
For such systems, Conditions~\ref{hyp:state_feedback} (local) and~\ref{hyp:distinguish} are proved to be sufficient to achieve the dynamic output feedback stabilization, which implies, by Theorem~\ref{th:necessary_common}, that Condition~\ref{hyp:common} is also satisfied.
In this paper, we therefore focus on systems that are not in the form of \eqref{E:diss}-\eqref{eq:diss}.
\end{rmrk}

\section{An illustrative example}\label{sec:finite}

\subsection{An obstruction by J.-M. Coron}\label{sec:coron}

Consider the case where \eqref{E:system_general} is single-input single-output and $f$ is a linear map, so that it can be written in the form of
\begin{equation}\label{E:system_linear}
\left\{
\begin{aligned}
&\dot{x}=Ax + bu,
\\
&y=h(x).
\end{aligned}
\right.
\end{equation}
where $A\in\R^{n\times n}$ and $b\in\R^{n\times 1}$ and $h:\R^n\to\R$.
If $h$ is nonlinear and is not an invertible transformation of a linear map, then the usual theory of linear systems fails to be applied.
Condition~\ref{hyp:state_feedback} reduces to the
stabilizability
of the pair $(A, b)$. If it holds, then \eqref{E:system_linear} is globally stabilizable by a linear static state feedback.

In \cite{Coron1994}, J.-M. Coron introduced the following illustrative one-dimensional example:
\begin{equation}\label{E:system_coron}
    \dot x = u,\qquad y = x^2.
\end{equation}
He proved that \eqref{E:system_coron} is not locally stabilizable by means of a dynamic output feedback, unless introducing a time-dependent component in the feedback law.
The difficulty with this system comes from the unobservability of the target point $0$. Indeed, \eqref{E:system_coron} is not observable for the constant input $u\equiv0$ in any time $T>0$. Indeed, the initial conditions $x_0$, $-x_0$ $\in\R$ are indistinguishable.
In particular, the system is not uniformly observable, and consequently the results of \cite{jouan, TeelPraly1994, TeelPraly1995} fail to be applied. To overcome this issue, \cite{Coron1994} introduced time-dependent output feedback laws, and proved by this means the local stabilizability of \eqref{E:system_coron}.
This system can also be stabilized by means of ``dead-beat'' or ``sample-and-hold'' techniques (see \cite{nevsic1998input}, \cite{ShimTeel2002}, respectively).

A generalization of \eqref{E:system_coron} in higher dimension is
\begin{equation}\label{E:system_rot}
\left\{
\begin{aligned}
&\dot{x}=Ax + bu,
\\
&y=h(x)%
\end{aligned}
\right.
\end{equation}
for a skew-symmetric matrix $A$ and $h$ radially symmetric\footnote{
Up to a change of scalar product, one may also consider the case where $PA+AP=0$ for some positive definite matrix $P\in\R^{n\times n
}$ and $h$ such that $(x_1'Px_1 = x_2'Px_2) \Rightarrow (h(x_1) = h(x_2))$.
}.
Again, the constant input $u\equiv0$ makes the system unobservable in any time $T>0$ since for any initial conditions $x_0$, $\tilde{x}_0$ in $\R^n$ satisfying $|x_0|=|\tilde{x}_0|$, $h(\varphi_t(x_0))=h(\tilde{x}_0)=h(x_0)=h(\varphi_t(x_0))$ for all $t\in\R_+$.
Condition~\ref{hyp:state_feedback}~(global) reduces to the stabilizability of $(A, b)$ and Condition~\ref{hyp:distinguish}~(global) is always satisfied.
Let us state a necessary condition for the stabilizability of \eqref{E:system_rot} by means of a dynamic output feedback.

\begin{thrm}\label{th:impossible}
If \eqref{E:system_rot} is locally stabilizable by means of a dynamic output feedback, then $A$ is invertible.
\end{thrm}
\begin{proof}
The proof is an adaptation of the one given in \cite{Coron1994} in the one-dimensional context.
Assume that
$(0, 0)$ is a locally asymptotically stable equilibrium point of
\begin{equation}\label{E:system_rot_stab}
\left\{
\begin{aligned}
&\dot{x}=\rot x+  bu,
\\
&y= h(x)
\end{aligned}
\right.
,\qquad
\left\{
\begin{aligned}
&\dot{\wdyn}=\nu(\wdyn, u, y)
\\
&u= \varpi(\wdyn, y)
\end{aligned}
\right.
\end{equation}
for some positive integer $q$ and two continuous maps $\nu:\R^q\times\R\times\R$ and $\varpi:\R^q\times\R$.
Set $F:\R^n\times\R^q\ni(x, \wdyn)\mapsto\left(\rot x+ b\varpi\left(\wdyn, h(x)\right), \nu\left(\wdyn, \varpi\left(\wdyn, h(x)\right), h(x)\right)\right)$.
Then, according to \cite[Theorem 52.1]{krasnosel1984geometrical}
(see~\cite{coron1994relations} when one does not have uniqueness of the solutions to the Cauchy problem),
the index of $-F$ at $(0, 0)$ is $1$.
Assume, for the sake of contradiction, that $A$ is not invertible.
Let $\NN$ be a one-dimensional subspace of $\ker A$.
Denote by $\Sigma$ the reflection through the hyperplane $\NN^\perp$,
that is, $\Sigma = \Id_{\R^n} - 2vv'$ for some unitary vector $v\in\NN$.
Then $\det \Sigma = -1$, $A\Sigma=A$ and $h(\Sigma x)=h(x)$.
Hence $(x, \wdyn)\mapsto -F(\Sigma x, \wdyn)$ has index $-1$ at $(0, 0)$ and $F(\Sigma x, \wdyn) = F(x, \wdyn)$. Thus $1=-1$ which is a contradiction.
\end{proof}

According to the spectral theorem, we have the following immediate corollary. If $n=1$, we recover the result of J.-M. Coron in \cite{Coron1994}.

\begin{crllr}\label{cor:impossible}
If $n$ is odd and $A$ is skew-symmetric, then \eqref{E:system_rot} is not locally stabilizable by means of a dynamic output feedback.
\end{crllr}

\subsection{Converse theorem: a positive result of output feedback stabilization}\label{sec:converse}

One of the main results of this paper is the following theorem which is the converse of Theorem~\ref{th:impossible} in the case where $h(x) = \frac{1}{2}|x|^2$. The proof relies on the guidelines described in the introduction, that is, an embedding into a bilinear system, an observer design with dissipative error-system and a feedback perturbation.

Consider the special case for system~\eqref{E:system_rot}:
\begin{equation}\label{E:system_rot2}
\left\{
\begin{aligned}
&\dot{x}=Ax + bu,
\\
&y=h(x)=\frac{1}{2}|x|^2.
\end{aligned}
\right.
\tag{\ref{E:system_rot}'}
\end{equation}

\begin{thrm}\label{th:finite}
If $A$ is skew-symmetric and invertible and $(A, b)$ is stabilizable, then \eqref{E:system_rot2} is semi-globally stabilizable by means of a dynamic output feedback.
\end{thrm}

\begin{rmrk}
The dynamic output feedback is explicitly given in \eqref{E:system_sep_rot}. It is easily implementable, and does not use time-dependent feedback laws.
\end{rmrk}

The proof of Theorem~\ref{th:finite} is the object of the section.
We follow the same steps as in \cite{ludo}, with a very similar embedding strategy.
The main difference is the observability analysis developped in Section~\ref{sec:obs_finie}:
here the target is unobservable,
while in \cite{ludo} it was observable.

\subsubsection{Embedding into a bilinear system of higher dimension}\label{sec:embedding-finite}
Consider the map
\fonction{\plong}{\R^n}{\R^{n+1}}{x}{\left(x, \frac{1}{2}\abs{x}^2\right).}
If $x$ is a solution of \eqref{E:system_rot},
then
$\frac{1}{2}\frac{\diff}{\diff t}\abs{x}^2 = x'Ax + x'bu = x'bu$ since $A$ is skew-symmetric.
Hence $\etat = \plong(x)$ defines an embedding of \eqref{E:system_rot} into
\begin{equation}\label{E:system_plonge_finie}
\left\{
\begin{aligned}
&\dot{\etat}=\mA(u)\etat + \mB u
\\
&y=\mC\etat.
\end{aligned}
\right.
\end{equation}
where
$\mA(u) =
\begin{pmatrix}
A    & 0\\
ub' & 0
\end{pmatrix}$,
$\mB =
\begin{pmatrix}
b\\
0
\end{pmatrix}$ and
$\mC = \begin{pmatrix}
0&\cdots&0&1
\end{pmatrix}$
and with initial conditions in $\TT=\plong(\R^n)$.
Moreover, the semi-trajectory $\etat$ remains in $\TT$.
We denote by $\inv:\R^{n+1}\to\R^n$ the projection operator given by 
$\etat =(\etat_1,\dots,\etat_{n+1}) \mapsto (\etat_1,\dots,\etat_{n})$.
Note that $\inv$ is a left-inverse of $\plong$:
\begin{equation}\label{def:uppi}
    \inv(\plong(x)) = x,\qquad \forall x\in\R^n.
\end{equation}
In the following, to ease notations, we often use the shorthand $\ubar{\etat}$ for $\inv(\etat)$.

\subsubsection{Observer design with dissipative error system}

Let us introduce a Luenberger observer with dynamic gain for \eqref{E:system_plonge_finie}.
In order to make the error system  dissipative,
set
$
\La(u)=\begin{pmatrix}
bu\\
\alpha
\end{pmatrix}\in\R^{n+1}
$ for some positive constant $\alpha$ to be fixed later.
The corresponding observer system is given by

\begin{equation}\label{E:system_observer_open}
\left\{
\begin{aligned}
&\dot{\varepsilon}= \left(\mA(u) -\La(u)\mC\right)\varepsilon
\\
&\dot{\zhat}= \mA(u) \zhat+  \mB u -\La(u)\mC\varepsilon
\end{aligned}
\right.
\end{equation}
where $\etat = \etath - \eps$ satisfies \eqref{E:system_plonge_finie}, $\etath$ is the estimation of the state made by the observer system and $\eps$ is the error between the estimation of the state and the actual state of the system.
Note that for all $u\in\R$,
\begin{equation}\label{E:dissipative}
\mA(u) - \La(u)\mC = 
\begin{pmatrix}
\rot & -bu\\
ub' & -\alpha
\end{pmatrix}
=
\begin{pmatrix}
\rot & -b u\\
ub' & 0
\end{pmatrix}
- \alpha \mC'\mC.
\end{equation}
It implies that the $\eps$-subsystem of \eqref{E:system_observer_open} is dissipative, that is,
for all input $u\in C^0(\R_+, \R)$, the solutions of \eqref{E:system_observer_open} satisfy
\begin{equation}\label{E:eps_decroit}
    \frac{\diff |\varepsilon|^2}{\diff t}
    = 2\varepsilon'\left(\mA(u) - \La(u)\mC\right)\varepsilon
    = -2\alpha |\mC\varepsilon|^2
    \leq 0.
\end{equation}
This is the first key fact of the strategy applied below.

\subsubsection{Feedback perturbation and closed-loop system}
\label{sec_FeedbackPert}

Because $(\rot, b)$ is stabilizable, there exists $K\in\R^{1\times n}$ such that $\rot+bK$ is Hurwitz (in particular, $(K, \rot)$ is detectable).
Since $\rot$ is skew-symmetric, its eigenvalues are purely imaginary. Hence, the Hautus lemmas for stabilizability (resp. detectability) and controllability (resp. observability) are equivalent.
Therefore, $(\rot, b)$ is controllable and $(K, \rot)$ is observable.

With a separation principle in mind, a natural strategy for dynamic output feedback stabilization of \eqref{E:system_rot} would be to combine the Luenberger observer~\eqref{E:system_observer_open} with the state feedback law $\phi:x\mapsto Kx$.
However, it appears that this strategy fails to be applied due to the unobservability at the target.
To overcome this difficulty, we rather consider a perturbed feedback law $\phi_\delta:x\mapsto Kx + \frac{\delta}{2}\abs{x}^2$ for some positive constant $\delta$ to be fixed later.
This is the second key fact of the strategy.
For all $\delta>0$, denote by $\Dd$ the basin of attraction of $0\in\R^n$ of the vector field $\R^n\ni x\mapsto Ax + b\phi_\delta(x)$.
Since the linearization of this vector field at $0$ is $x\mapsto (A+bK)x$, it is locally asymptotically stable at $0$ for all $\delta>0$.
As stated in the following lemma, the drawback of this perturbation is to pass from a globally stabilizing state feedback to a semi-globally stabilizing one.
\begin{lmm}\label{lem:delta}
For any compact set $\Kx\subset\R^n$, there exists $\delta_0>0$ such that for all $\delta\in(0, \delta_0)$, $\Kx\subset\Dd$.

\end{lmm}
\begin{proof}
Let $\rho>0$ be such that $\Kx\subset B_{\R^n}(0, \rho)$.
Since $A+bK$ is Hurwitz, there exists $P\in\R^{n\times n}$ positive definite such that 
$P(A+bK) + (A+bK)'P < -2 \Id_{\R^n}$
(recall that $'$ denotes the transpose operation).
Set $V:\R^n\ni x\mapsto x'Px$.
Then, for all $x\in\Kx$,
\begin{align*}
    \frac{\partial V}{\partial x}(x)(Ax+b\phi_\delta(x))
    &=2x'P(A+bK)x + \delta |x|^2 x'Pb \\
    &\leq (-2+ \delta |x||Pb|) |x|^2\\
    &\leq (-2+ \delta \rho |Pb|) |x|^2.
\end{align*}
Set $\delta_0 = \frac{1}{\rho |Pb|}$ and let $\delta\in(0, \delta_0)$.
Then $V$ is positive definite and 
\begin{align*}
    \frac{\partial V}{\partial x}(x)(Ax+b\phi_\delta(x))
    < -|x|^2
\end{align*}
for all $x\in\Kx$.
Hence, $0\in\R^n$ is a locally asymptotically (even exponentially) stable equilibrium point of the vector field $\R^n\ni x\mapsto Ax + b\phi_\delta(x)$ with basin of attraction containing $\Kx$.
\end{proof}

Hence, for all compact set $\Kx\subset\R^n$ there exists $\delta_0>0$ such that if $\delta\in(0, \delta_0)$, then $\Kx\subset\Dd$.
On system~\eqref{E:system_plonge_finie}, we choose the feedback law
\begin{equation}
\lambda_\delta(z)=
\begin{pmatrix}
K& \delta
\end{pmatrix}z,
\end{equation}
which satisfies $\phi_\delta=\lambda_\delta\circ\plong$.
The corresponding closed-loop system is given by
\begin{equation}\label{E:system_observer_closed}
\left\{
\begin{aligned}
&\dot{\varepsilon}= \left(\mA(\lambda_\delta(\zhat)) -\La(\lambda_\delta(\zhat))\mC\right)\varepsilon,
\\
&\dot{\zhat}= \mA(\lambda_\delta(\zhat)) \zhat+  \mB \lambda_\delta(\zhat) -\La(\lambda_\delta(\zhat))\mC\varepsilon.
\end{aligned}
\right.
\end{equation}
We are now able to exhibit a coupled system in the form of~\eqref{E:system_stab} (with $w=\etath$) with which we intend to prove semi-global dynamic output feedback stabilization of~\eqref{E:system_rot}:
\begin{equation}\label{E:system_sep_rot}
\left\{
\begin{aligned}
&\dot{x}=\rot x+  bu,
\\
&y= \frac{1}{2}\abs{x}^2
\end{aligned}
\right.
,\qquad
\left\{
\begin{aligned}
&\dot{\zhat}=\mA(u) \zhat+  \mB u -\La(u)\left(\mC\zhat-y\right)
\\
&u= \lambda_\delta(\zhat).
\end{aligned}
\right.
\end{equation}

It is now sufficient to prove the following theorem, which implies Theorem~\ref{th:finite}, in the next sections.
\begin{thrm}\label{th:explicit}
For all compact set $\Kx\times\Kw\subset\R^n\times\R^{n+1}$, there exist $\delta_0>0$ and $\alpha_0>0$ such that
for all $\delta\in(0, \delta_0)$
and all $\alpha\in(\alpha_0, +\infty)$,
$(0, 0)\in\R^n\times\R^{n+1}$ is a locally asymptotically stable equilibrium point of \eqref{E:system_sep_rot} with basin of attraction containing $\Kx\times\Kw$.
\end{thrm}

\subsubsection{Boundedness of trajectories}

Since $\R^n\ni x\mapsto \frac{1}{2}|x|^2$ and $\phi_\delta$ are locally Lipschitz continuous functions, according to the Cauchy-Lipschitz theorem, for any initial condition $(x_0, \etath_0)\in\R^n\times\R^{n+1}$, there exists exactly one maximal solution $(x, \etath)$ of \eqref{E:system_sep_rot} such that $(x(0), \etath(0)) = (x_0, \etath_0)$.
Before going into the proof of Theorem~\ref{th:explicit}, we need to ensure the existence of global solutions.

\begin{lmm}\label{lem:bound}
For any compact set $\Kx\times\Kw\subset\R^n\times\R^{n+1}$,
there exist $\delta_0>0$ and $\alpha_0>0$ such that
for all $\delta\in(0, \delta_0)$
and all $\alpha\in(\alpha_0, +\infty)$,
\eqref{E:system_sep_rot} has a unique global solution $(x, \etath)$ for each initial condition $(x_0, \etath_0)\in\Kx\times\Kw$.
Moreover, $(x, \etath)$ is bounded and $\ubar{\etath}$ remains in a compact subset of $\Dd$.
\end{lmm}
\begin{proof}
Let $(x_0, \etath_0)\in\Kx\times\Kw$ and $(x, \etath)$ be the corresponding maximal solution of \eqref{E:system_sep_rot}. Set $\etat=\plong(x)$ and $\eps=\etath-\etat$, so that $(\eps, \etath)$ is the maximal solution of \eqref{E:system_observer_closed} starting from $(\eps_0, \etath_0)$.
Then, it is sufficient to prove that $(\eps, \etath)$ is a global solution, $(\eps, \etath)$ is bounded and $\ubar{\etath}$ remains in a compact subset of $\Dd$.
According to \eqref{E:eps_decroit}, $\eps$ is bounded since $|\eps|$ is non-increasing.
Moreover, $\etath_{n+1} = \eps_{n+1} + \frac{1}{2}|\ubar{\etat}|^2 = \eps_{n+1} + \frac{1}{2}|\ubar{\etath}-\ubar{\eps}|^2$.
Then, it remains to show that
there exist $\delta_0>0$ and $\alpha_0>0$ such that for all $\delta\in(0, \delta_0)$
and all $\alpha\in(\alpha_0, +\infty)$,
for all initial conditions $(\eps_0, \etath_0)\in\Ke\times\Kw$, $\ubar{\etath}$ remains in a compact subset of $\Dd$.

Since $A+bK$ is Hurwitz, there exists $P\in\R^{n\times n}$ positive definite such that $P(A+bK) + (A+bK)'P < -2 \Id_{\R^n}$.
Then $V:\R^n\ni x\mapsto x'Px$ is a strict Lyapunov function for system~\eqref{E:system_rot} with feedback law $\phi$.
For all $r>0$, set $D(r)=\{x\in\R^n\mid V(x)\leq r\}$.
Let $\rho'>\rho>0$ and $r'>r>0$ be such that
$B_{\R^{n+1}}(0,\rho)$ contains $\Ke$ and $\Kw$
and $B_{\R^{n+1}}(0,\rho)\subset D(r)\subset D(r')\subset B_{\R^{n+1}}(0,\rho')$.
According to Lemma~\ref{lem:delta}, there exists $\delta_0>0$ such that for all $\delta\in(0, \delta_0)$, $\Dd$ contains the closure of $B_{\R^n}(0,\rho')$.
In the following, we show that there exists $\alpha_0>0$ such that, if $\alpha>\alpha_0$, then $\ubar{\etath}$ remains in $B_{\R^n}(0,\rho')$.
For all $\zhat, \eps$ in $\R^{n+1}$, define
\begin{align*}
\mu^1_\delta(\zhat)
&=\mA(\phi_\delta(\ubar{\etath}))\zhat+\mB\phi_\delta(\ubar{\zhat}),\\
\mu^2_\delta(\zhat)
&=(\mA(\lambda_\delta(\zhat))-\mA(\phi_\delta(\ubar{\etath})))\zhat+\mB(\lambda_\delta(\zhat)-\phi_\delta(\ubar{\etath})),\\
\mu^3_{\delta, \alpha}(\varepsilon,\zhat)&
=-\La(\lambda_\delta(\zhat))\mC\varepsilon,
\end{align*}
so that the solutions of \eqref{E:system_observer_closed} satisfy
\begin{equation}
    \dot{\zhat}=\mu^1_\delta(\zhat)+\mu^2_\delta(\zhat)+\mu^3_{\delta, \alpha}(\varepsilon,\zhat).
\end{equation}
In particular,
$$
\dot{\ubar{\etath}}
=
A\ubar{\etath} + \lambda_\delta(\zhat) b -  \lambda_\delta(\zhat) \eps_{n+1} b.
$$
By continuity of 
$(\zhat, \delta)\mapsto\lambda_\delta(\zhat)$, %
$$
\overline{M}:=\sup_{\substack{
\varepsilon,\zhat\in B_{\R^{n+1}}(0,\rho')
\\
\delta\in[0, \delta_0]}
}
|A\ubar{\etath} + \lambda_\delta(\zhat) b -  \lambda_\delta(\zhat)\eps_{n+1} b|
<\infty.
$$
Let $T_0 =\frac{\rho'-\rho}{\overline{M}}$. Since $|\varepsilon|$ is non-increasing,
any trajectory of \eqref{E:system_observer_closed} starting in $B_{\R^{n+1}}(0,\rho)\times B_{\R^{n+1}}(0,\rho)$ will be such that
$\ubar{\zhat}$ remains in $B_{\R^{n}}(0,\rho')$ over the time interval $[0,T_0]$.
It remains to show that $\ubar{\etath}$ does not exit $B_{\R^n}(0,\rho')$ after time $T_0$.

Note that $\mu^1_\delta(\zhat_1) = \mu^1_\delta(\zhat_2)$ if $\inv(\zhat_1)=\inv(\zhat_2)$.
Then,
$$
\underline{m}:=
- \max_{
\substack{
		\ubar{\etath}\in\partial D(r')
		\\
		\zhat\in B_{\R^{n+1}}(0,\rho')
		}
}
\big(L_{\mu_0^1}V\circ\inv\big)(\zhat)
= - \max_{
\substack{
		\inv(\etath)\in\partial D(r')
		\\
		\zhat\in B_{\R^{n+1}}(0,\rho')
		}
}
\frac{\partial V}{\partial x}\left(\inv(\etath)\right)(A+bK)\inv(\etath)
>0.
$$
Notice that
$
(\mu^1_\delta-\mu^1_0+\mu^2_\delta)(\zhat)
=
\delta \zhat_{n+1}
\begin{pmatrix}
b\\
b'\ubar{\zhat}
\end{pmatrix}$.
Hence, without loss of generality, one can assume that  $\delta_0>0$ is (small enough) such that for all $\delta\in(0, \delta_0)$,
$$
\max_{
B_{\R^{n+1}}(0,\rho')
}
|L_{\mu^1_\delta-\mu^1_0+\mu^2_\delta}V\circ\inv|\leq \frac{1}{3}\underline{m}.
$$
Fix $\delta\in(0, \delta_0)$.
Assume for the sake of contradiction that $\ubar{\etath}$ leaves $D(r')$ for the first time at $T'_0>T_0$. Then
\begin{align*}
0
&\leq
\frac{\diff}{\diff t}V\left(\inv(\zhat(t))\right)\Big|_{t=T'_0}\\
&=
(L_{\mu_0^1}V\circ\inv)(\zhat(T'_0))
+
(L_{\mu^1_\delta-\mu^1_0+\mu^2_\delta}V\circ\inv)(\zhat(T'_0))
+ %
\frac{\partial V\circ\inv}{\partial \zhat}(\zhat(T'_0))\mu^3_{\delta, \alpha}(\eps(T'_0), \zhat(T'_0))
\\
&\leq - \frac{2}{3}\underline{m}
+
\frac{\partial V\circ\inv}{\partial \zhat}(\zhat(T'_0))\mu^3_{\delta, \alpha}(\eps(T'_0), \zhat(T'_0))
\end{align*}
Now, we show that there exists $\alpha_0>0$ big enough such that for all $\alpha>\alpha_0$,
\begin{equation}\label{E:size_pert_h}
\frac{\partial V\circ\inv}{\partial \zhat}(\zhat(T'_0))\mu^3_{\delta, \alpha}(\eps(T'_0), \zhat(T'_0))
\leq \frac{1}{3}\underline{m},
\end{equation}
which contradicts $\underline{m}>0$.
By definition of $\La$, $\inv$ and $\mu^3_{\delta, \alpha}$,
$$
\frac{\partial V\circ\inv}{\partial \zhat}(\zhat)\mu^3_{\delta, \alpha}(\eps, \zhat)
=
-
\varepsilon_{n+1}
\lambda_\delta(\zhat)
\frac{\partial V}{\partial x} (\inv(\zhat)) b.
$$
Let $Q=\max_{
\substack{
		(\zhat_2,\zhat_3)\in\partial B_{\R^n}(0,\rho')
		\\
		\varepsilon,\zhat\in B_{\R^{n+1}}(0,\rho')
		}
} |\lambda_\delta(\zhat)  \frac{\partial V}{\partial x} (\inv(\zhat)) b|
$,
so that
$|\lambda_\delta(\zhat(T'_0))  \frac{\partial V}{\partial x} (\inv(\zhat(T'_0))) b|\leq Q$.
Recall that
$$
\dot{\varepsilon}_{n+1}=-\alpha \varepsilon_{n+1}+\lambda_\delta(\zhat)b'\ubar{\varepsilon}
$$
and thus, for all $t\geq0$,
$$
\varepsilon_{n+1}(t)=\e^{-\alpha t}\varepsilon_{n+1}(0)+\int_{0}^{t}\e^{-\alpha(t-s)}\lambda_\delta(\zhat(s))b'\ubar{\varepsilon}(s)\diff s.
$$
Moreover,
$\varepsilon(t)$ and $\ubar{\zhat}(t)$ are in $B_{\R^{n+1}}(0, \rho')$ for all $t\in[0, T'_0]$ and
$$
\lambda_\delta(\zhat)=\begin{pmatrix}
K&\delta
\end{pmatrix}
\zhat
=
K \ubar{\zhat}
+
\delta\left(\eps_{n+1} + \frac{1}{2}|\ubar{\etath}-\ubar{\eps}|^2\right)
.
$$
Hence,
$$
|\lambda_\delta(\zhat)|
\leq 
\rho'
\left(|K|
+
\delta(1+2\rho')
\right).
$$
As a consequence, for all $t\in[0, T'_0]$,
$$
|\varepsilon_{n+1}(t)|\leq \rho' \left(\e^{-\alpha t}+\frac{\rho'^2|b|}{\alpha}\left(|K|
+
\delta(1+2\rho')
\right)\right) .
$$
Thus there exists  $\alpha_0>0$ such that if $\alpha>\alpha_0$,
then $|\varepsilon_{n+1}(T'_0)|\leq \dfrac{\underline{m}}{3 Q}$.
Fix $\alpha>\alpha_0$. Then \eqref{E:size_pert_h} holds, which concludes the proof of the lemma.
\end{proof}

\subsubsection{Observability analysis}\label{sec:obs_finie}

The following lemma is a crucial step of the proof of Theorem~\ref{th:finite} that emphasizes the usefulness of the feedback perturbation described above. Indeed, one can easily see that its proof fails if $\delta=0$ (since the matrix $\mQ$ defined below is not invertible in this case).

\begin{lmm}\label{lem:observability}
Let $(z_0,\zhat_0)\in \left(\TT\times \R^{n+1}\right)\setminus\{(0,0)\}$. Let $(\varepsilon,\zhat)$ be the semi-trajectory of \eqref{E:system_observer_open} with initial condition $(\zhat_0-z_0,\zhat_0)$. Then, for all $T>0$, \eqref{E:system_plonge_finie} is observable in time $T$ for the input $u=\lambda_\delta(\zhat)$.
\end{lmm}

\begin{proof}
Let $\omega_0\in\ker(\mC)\setminus\{0\}$,
and consider $\omega$ a solution of  the dynamical system
\begin{equation}\label{E:omega}
   \dot{\omega}=\mA(\lambda_\delta(\zhat))\omega
\end{equation}
with initial condition
$\omega_0$.
To prove the result, it is sufficient to show that $\mC\omega$ has a non-zero derivative of some order at $t=0$ if $(\eps_0,\zhat_0)\neq(0, 0)$.
Indeed, it implies that for all initial conditions $\etat_0\neq\tilde{\etat}_0$ in $\R^{n+1}$, if $\etat$ (resp. $\tilde{\etat}$) is the solution of \eqref{E:system_plonge_finie} with initial condition $\etat_0$ (resp. $\tilde{\etat}_0$),
then $\omega = \etat - \tilde{\etat}$ is a solution to \eqref{E:omega} starting at $\omega_0\neq0$ and $\mC\omega$
is not constantly equal to zero on any time interval $[0, T]\subset\R_+$.
We prove this fact by contradiction: assume that 
\begin{equation}\label{E:omega_der}
\mC\omega^{(k)}(0)=\omega_{n+1}^{(k)}(0)=0\qquad \forall k\in \N, 
\end{equation}
for some $\omega(0)\neq0$, and prove that $(z_0, \zhat_0) = (0, 0)$.
Let $u = \lambda_\delta(\zhat)$. Then $\dot\omega_{n+1}=ub'\ubar{\omega}$ and $\dot{\ubar{\omega}} = \rot\ubar{\omega}$.
Hence
\begin{equation}
    \label{leibniz}
    0 = \omega_{n+1}^{(k+1)}(0) = \sum_{i=0}^k \binom{k}{i} u^{(i)}(0) b'\rot^{k-i}\ubar{\omega}(0)
\end{equation}
for all $k\in\N$, where $\binom{k}{i}$ denote binomial coefficients.
The proof goes through the following three steps.

\textbf{Step 1: Show that $\boldsymbol{u^{(k)}(0)=0}$ for all $\boldsymbol{k\in\N}$.}
Let $p\in\N$ be the smallest integer such that $u^{(p)}(0)\neq0$ and look for a contradiction. Equation \eqref{leibniz} yields
\begin{equation}
    \label{leibniz2}
    \sum_{i=0}^k \binom{p+k}{p+i} u^{(p+i)}(0) b'\rot^{k-i}\ubar{\omega}(0) = 0
\end{equation}
for all $k\in\N$. 
Since $(\rot, b)$ is controllable and $\ubar{\omega}(0)\neq0$, there exists $q\in\{0,\dots,n\}$ such that $ b'\rot^{q}\ubar{\omega}(0)\neq0$
and
$b'\rot^{i}\ubar{\omega}(0)=0$
for all $i\in\{0,\dots,q-1\}$.
Then
\begin{equation}
    \label{leibniz3}
    0 = \sum_{i=0}^q \binom{p+q}{p+i} u^{(p+i)}(0)  b'\rot^{q-i}\ubar{\omega}(0)
    =\binom{p+q}{p} u^{(p)}(0)  b'\rot^{q}\ubar{\omega}(0).
\end{equation}
which is a contradiction.

\textbf{Step 2: Find $\boldsymbol{\mQ\in\R^{(n+2)\times(n+2)}}$ (invertible) such that $\boldsymbol{\mQ\begin{pmatrix}
\zhat(0)\\
\epsilon_{n+1}(0)
\end{pmatrix}=0}$.}
For all $k\in\N$,
$$
0 = u^{(k)}(0)
= \begin{pmatrix}
K& \delta
\end{pmatrix}\zhat^{(k)}(0).
$$
Moreover,
\begin{align*}
\begin{pmatrix}
\dot{\ubar{\zhat}}\\
\dot{\zhat}_{n+1}\\
\dot{\epsilon}_{n+1}
\end{pmatrix}
=
\begin{pmatrix}
\rot & -bu & 0\\
b'u & 0 & -\alpha\\
0 & 0 & -\alpha
\end{pmatrix}
\begin{pmatrix}
\ubar{\zhat}\\
\zhat_{n+1}\\
\epsilon_{n+1}
\end{pmatrix}
+ u
\begin{pmatrix}
b\\
0\\
b' \ubar{\eps}
\end{pmatrix}.
\end{align*}
Hence, for all $k\geq1$,
$
\ubar{\zhat}^{(k)}(0) = \rot^k\ubar{\zhat}(0)
$ and
$
\zhat_{n+1}^{(k)}(0) = \epsilon_{n+1}^{(k)}(0) = (-\alpha)^k\epsilon_{n+1}(0)
$.\\
Thus
$\begin{pmatrix}
K\rot^k& \delta(-\alpha)^{k}
\end{pmatrix}\begin{pmatrix}
\ubar{\zhat}(0)\\
\epsilon_{n+1}(0)
\end{pmatrix} = 0$ for all $k\geq1$. By setting
\begin{equation}
\mQ =
\begin{pmatrix}
K & \delta & 0\\
K \rot & 0 & - \delta \alpha\\
\vdots & \vdots & \vdots \\
K \rot^{n+1} & 0 & \delta(-\alpha)^{n+1}
\end{pmatrix}
\end{equation}
we get that
$\mQ\begin{pmatrix}
\zhat(0)\\
\epsilon_{n+1}(0)
\end{pmatrix}=0$.

\textbf{Step 3: Conclusion.}
In Appendix~\ref{app:det}, we check that $\mQ$ is invertible.
Hence, $\zhat(0) = 0$ and $\eps_{n+1}(0) = 0$. Thus, $\frac{1}{2}\abs{\ubar{z}(0)}^2 = z_{n+1}(0) = \zhat_{n+1}(0) - \eps_{n+1}(0) = 0$ \emph{i.e.} $(z_0, \zhat_0) = (0, 0)$ which is a contradiction.
\end{proof}

On the basis of Lemmas~\ref{lem:bound} and~\ref{lem:observability},
we are now in position to prove Theorem~\ref{th:explicit}.
Let $\Kx\times\Kw\subset\R^n\times\R^{n+1}$ be a compact set,
and $\delta_0>0$ and $\alpha_0>0$ be as in Lemma~\ref{lem:bound}.
Fix $\delta\in(0, \delta_0)$ and $\alpha\in(\alpha_0, +\infty)$.
Let $(x_0, \etath_0)\in\Kx\times\Kw$ and $(x, \etath)$ be the corresponding solution of \eqref{E:system_sep_rot}. Set $\etat=\plong(x)$, $\eps=\etath-\etat$ so that $(\eps, \etath)$ is the solution of \eqref{E:system_observer_closed} starting from $(\eps_0, \etath_0)$,
$\eps_0=\etath_0-\plong(x_0)$.
We need to show the two following statements:
\medskip
\begin{itemize}[wide=0pt, labelwidth=\widthof{2. (Attractivity)~}]
    \item[1. (Stability)\hfill] $(0, 0)$ is a stable equilibrium point of~\eqref{E:system_sep_rot},
    \item[2. (Attractivity)] and its basin of attraction contains $\Kx\times\Kw$.
\end{itemize}
\medskip
We prove the former in Section~\ref{sec:stability} and the latter in Section~\ref{sec:attractivity}.

\subsubsection{Stability}\label{sec:stability}

Let $R>0$. We seek $r>0$ such that, if $|x_0|, |\etath_0|\leq r$, then $|x(t)|, |\etath(t)| \leq R$ for all $t\in\R_+$. We have
\begin{align*}
    \dot x
    &= A x + b \lambda_\delta(\etath)
    \\
    &= A x + b \lambda_\delta(\plong(x)+\eps)
    \\
    &= Ax + b\phi_\delta(x)
    + b \begin{pmatrix}
    K&\delta
    \end{pmatrix}
    \eps.
\end{align*}
Fix $\eta>0$ 
such that %
$R-\eta\sqrt{1+\frac{\eta^2}{2}}>0$.
Since $x\mapsto Ax + b\phi_\delta(x)$ is locally asymptotically stable,
there exist
two positive constant 
$r_x<\eta$ and
$\re\leq R-\eta\sqrt{1+\frac{\eta^2}{2}}$ such that,
if 
$|x_0|\leq r_x$ and
$|\eps(t)|\leq \re$ for all $t\in\R_+$, then $|x(t)|\leq \eta$ for all $t\in\R_+$.
Let 
$r\in(0, r_x)$
be such that
$r+r\sqrt{1+\frac{r^2}{2}}\leq\re$.
Assume that $|x_0|, |\etath_0|\leq r$.
Then,
$$
|\eps_0|
\leq |\etath_0|+|\plong(x_0)|
= |\etath_0|+|x_0|\sqrt{1+\frac{|x_0|^2}{2}}
\leq r + r\sqrt{1+\frac{r^2}{2}}
\leq \re.
$$
According to \eqref{E:eps_decroit}, $|\eps|$ is non-increasing.
Hence, 
for all $t\in\R_+$,
$|x(t)|\leq\eta\leq R$ and
$$|\etath(t)|\leq |\plong(x(t))| + |\eps(t)|\leq\eta\sqrt{1+\frac{\eta^2}{2}}+\re\leq R.$$

\subsubsection{Attractivity}\label{sec:attractivity}

According to \eqref{E:eps_decroit},
$\frac{\diff |\varepsilon|^2}{\diff t}= -2\alpha |\mC\varepsilon|^2$.
Due to LaSalle's invariance principle, the $\omega$-limit set of $\eps$ is the largest invariant subset of $\ker \mC$.
Since $\eps$ satisfies~\eqref{E:omega} on this set, Lemma~\ref{lem:observability} guarantees that either $\eps\equiv0$, or $(\eps_0, \etath_0)=(0, 0)$, which also implies $\eps\equiv0$.
Therefore, the $\omega$-limit set of $\eps$ reduces to $\{0\}$, \ie $\eps\to0$.

Since $\etath_{n+1} = \eps_{n+1} + \frac{1}{2}|\ubar{\etath}-\ubar{\eps}|^2$, it remains to prove that $\ubar{\etath}\to0$.
First, notice that 
$$
|\mu^2_\delta(\zhat)|
=
|\lambda_\delta(\zhat)-\phi_\delta(\ubar{\zhat})|\sqrt{|b|^2+|b'\ubar{\etath}|^2}
$$
and
$$
|\mu^3_{\delta, \alpha}(\varepsilon,\zhat)|
=
\sqrt{\alpha^2+|b|^2\lambda_\delta(\zhat)^2}|\mC\varepsilon|.
$$
Since $C\varepsilon\to 0$ and $\zhat$ is bounded, $|\mu^3_{\delta, \alpha}(\varepsilon,\zhat)|\to 0$.
Likewise,
\begin{align*}
\lambda_\delta(\zhat)-\phi_\delta(\zhat)
&=\delta\left(\etath_{n+1}-\frac{1}{2}|\ubar{\etath}|^2\right)\\
&=\delta\left(\eps_{n+1}+\etat_{n+1}-\frac{1}{2}|\ubar{\eps}|^2-\frac{1}{2}|\ubar{\etat}|^2+\ubar{\eps}'\ubar{\etat}\right)\\
&= \delta\left(\eps_{n+1}-\frac{1}{2}|\ubar{\eps}|^2+\ubar{\eps}'\ubar{\etat}\right).
\end{align*}
Since $\varepsilon\to 0$ and $\etat$ is bounded, $\mu^2_\delta(\zhat)\to0$.

According to the converse Lyapunov theorem \cite[Theorem 1]{teel2000smooth}, there exists a strict proper Lyapunov function $V_\delta$ for
system~\eqref{E:system_rot} with feedback law $\phi_\delta:x\mapsto K x+\frac{\delta}{2}|x|^2$ over the basin of attraction $\Dd$.
For all $r>0$, set $D(r)=\{x\in\Dd\mid V_\delta(x)\leq r\}$.
In order to prove that $\ubar{\etath}\to0$, we show that for all $r>0$, there exists $T(r)\geq0$ such that $\ubar{\etath}(t)\in D(r)$ for all $t\geq T(r)$.
According to Lemma~\ref{lem:bound}, there exists a compact set $\Kwn\subset \Dd$ such that $\ubar{\etath}\in \Kwn$.
If $r>0$ is such that $\Kwn\subset D(r)$ then $T(r)=0$ satisfies the statement.
Let $0<r<R$ be such that $\Kwn\not \subset  D(r)$ and $\Kwn\subset D(R)$, then
$$
\bar{m}:=-\max_{x\in D(R)\setminus D(r)} 
\frac{\partial V}{\partial x}(x)(Ax+b\phi_\delta(x))
>0.
$$
Since $|g(\zhat(t))|\to 0$ and $|h(\varepsilon(t),\zhat(t))|\to 0$, there exists $T_1(r)>0$ such that for all $t\geq T_1(r)$, if $\ubar{\etath}(t)\not\in D(r)$, then
$$
\frac{\diff}{\diff t}V_\delta(\ubar{\etath})<-\frac{\bar{m}}{2}.
$$
First, this implies that if $\ubar\zhat(t)\in D(r)$ for some $t\geq T_1(r)$, then $\ubar\zhat(s)\in D(r)$ for all $s\geq t$.
Second, for all $t\geq0$,
\begin{align*}
    V_\delta(\ubar{\etath}(T_1(r)+t))
    &=
    V_\delta(\ubar{\etath}(T_1(r)))
    + \int_0^{t}\frac{\diff}{\diff s}V_\delta(\ubar{\etath}(T_1(r)+s))\diff s
    \\
    &\leq R - \frac{\bar{m}}{2}t \tag{while $\ubar{\etath}(T_1(r)+t)\notin D(r).$}
\end{align*}
Set $T_2(r) = \frac{2R-r}{\bar{m}}$ and $T(r) = T_1(r) + T_2(r)$. Then for all $t\geq T(r)$, $\ubar{\etath}(t)\in D(r)$, which concludes the proof of convergence, and therefore the proof of Theorem~\ref{th:explicit}.

\section{An infinite-dimensional perspective}\label{sec:infinite}

Guided by the illustrative example of Section~\ref{sec:finite}, we aim to provide more general results, based on the same two principles: embedding into a dissipative system, and feedback perturbation.
The embedding strategy used in Section~\ref{sec:embedding-finite} appears to be hardly generalizable, a different strategy must be found.
In \cite{Celle-etal.1989}, the authors introduce a  technique for the synthesis of observers for nonlinear systems.
The method is based on representation theory, and embedding into bilinear unitary systems.
It is far more general than the embedding found in Section~\ref{sec:embedding-finite}.
The price to pay is that the observer system can be infinite-dimensional.
In the rest of the paper,
we apply this strategy in the context of dynamic output feedback stabilization at an unobservable target.
In this section, we exhibit
some general results when such an embedding exists.

\subsection{Infinite-dimensional framework and statement of the main result}

Since our goal is to introduce an infinite-dimensional strategy for output feedback stabilization,
Definition~\ref{def:stab_out} must be amended.
When dealing with infinite-dimensional systems, it is necessary to fix a suitable functional framework.
Moreover, we also would like take into account piecewise constant feedback laws.
For these reasons, we introduce the following definitions.

\begin{dfntn}\label{def:framework}
Let $(\XX,\|\cdot \|_{\XX})$ be a Hilbert space and $\dom\subset\XX$ be a dense subspace.
Let $(\EE(u):\dom\to\XX)_{u\in\R^p}$ be a family of 
(potentially) unbounded linear operators.
Let $\BB:\R^m\to\XX$ and $\varpi:\XX\times\R^m\to\R^p$ be two continuous maps.
For all $k\in\N$, set $t_k = k\dt$ for some $\dt>0$. We call \emph{infinite-dimensional piecewise constant dynamic output feedback} of system~\eqref{E:system_general} the system
\begin{equation}\label{E:system_stab_infinie_0}
\left\{
\begin{aligned}
&\dot x = f(x, u)\\
&y = h(x)
\end{aligned}
\right.
,\qquad
\left\{
\begin{aligned}
&\dot{\etath}= \EE(u) \etath + \BB(y)\\
&u(t_k) = \varpi(\etath(t_k^-), y(t_k^-))\\
&u(t) = u(t_k),\qquad t\in[t_k,t_{k+1})
\end{aligned}
\right.
\end{equation}
where %
$\varpi(\etath(t_k^-), y(t_k^-)) = \lim_{\substack{t\to t_k\\t< t_k}}\varpi(\etath(t), y(t))$.

\end{dfntn}

\begin{dfntn}\label{def:stab_inf}
Let $\Kx\subset\R^n$ be a compact set.
System~\eqref{E:system_general}
is said to be \emph{stabilizable over $\Kx$ by means of an infinite-dimensional 
piecewise constant
dynamic output feedback}
if there exists a feedback in the form of \eqref{E:system_stab_infinie_0} as in Definition~\ref{def:framework}, a bounded set $\Kz\subset\XX$ and $\etath^\star\in\Kz$ such that the following conditions are satisfied:
\begin{enumerate}[label=\textit{(\roman*)}]
    \item For all $(x_0, \etath_0)\in\Kx\times\Kz$,
    there exists at least one solution $(x, \etath)$
    of \eqref{E:system_stab_infinie_0}
    in
    $C^0(\R_+, \R^n\times\doma)$.
    \label{defi}
    \item For all $R_x, R_{\etath}>0$, there exist $r_x, r_{\etath}>0$ such that
    for all $(x_0, \etath_0)\in\Kx\times\Kz$,
    if $|x_0|<r_x$ and $\norm{\etath_0-\etath^\star}<r_{\etath}$,
    then any solution $(x, \etath)$ of \eqref{E:system_stab_infinie_0}
    starting from $(x_0, \etath_0)$ satisfies
    $|x(t)|<R_x$ and $\norm{\etath(t)-\etath^\star} < R_{\etath}$ for all $t\geq0$.
    \label{defii}
    \item Any solution $(x, \etath)$ of \eqref{E:system_stab_infinie_0} with initial condition in $\Kx\times\Kz$ is such that
    $x(t)\to0$ and $\etath(t)\cvf\etath^\star$ as $t$ goes to infinity.
    \label{defiii}
\end{enumerate}
Furthermore, this property is said to be \emph{semi-global} if it holds for any compact $\Kx\subset \R^n$.

\end{dfntn}

\begin{rmrk}
If $\XX$ is finite-dimensional, then \ref{defi}-\ref{defii}-\ref{defiii} is equivalent to the usual definition of asymptotic stability of~\eqref{E:system_stab_infinie_0} at $(0, 0)$ with basin of attraction containing $\Kx\times\Kz$
(except that the feedback is now piecewise constant).
However, when $\XX$ is infinite-dimensional (the case of interest
in this section), the convergence of trajectories towards the equilibrium point holds only in the weak topology. Hence, \ref{defi}-\ref{defii}-\ref{defiii} is not equivalent to the usual definition of asymptotic stability of the infinite-dimensional system~\eqref{E:system_stab_infinie_0}.
\end{rmrk}

In the following three sections, we introduce some general tools that can be used to prove stabilizability by means of an infinite-dimensional 
piecewise constant
dynamic output feedback.
Their development is motivated by the investigation of 
the following two-dimensional system presenting an 
archetypal singularity at the target point:

\begin{equation}\label{E:system_depart0}
\left\{
\begin{aligned}
&\xdot= \matR x + b u
\\
&y= h(x)
\end{aligned}
\right.
\qquad
\text{with}\
\matR = \begin{pmatrix}
0&-1
\\
1&0
\end{pmatrix},
\
b = \begin{pmatrix}
0
\\
1
\end{pmatrix}
\ \text{and}\
h:\R^2\to\R^m.
\end{equation}

Note that \eqref{E:system_depart0} is in the form of~\eqref{E:system_rot} as soon as $h$ is  radially symmetric.
Of course, our analysis is of interest only if~\eqref{E:system_depart0} is not uniformly observable. 
In particular, with $h(x) = \frac{1}{2}|x|^2$, we recover a subcase of \eqref{E:system_rot2} that was investigated in Theorem~\ref{th:finite}.
The finite-dimensional strategy developed in Section~\ref{sec:finite}
was specific to this output, and we wish to explore stabilization for different output maps that present other observability singularities.

We can now state the main stabilization result obtained on system~\eqref{E:system_depart0}, which is the main matter of  Section~\ref{sec:example_infinie}.
The method relies on unitary irreducible representations of the group induced by the dynamics. Their corresponding special functions, which are Bessel functions, play a major role.
Recall that the Bessel function of the first kind of order $k\in\Z$ is given by:
\begin{equation}\label{Bessel_def}
    J_k:\R\ni \ray \mapsto \frac{1}{2\pi}\int_0^{2\pi}e^{i\ray\sin(\varS) - ik\varS}\diff \varS\in\R.
\end{equation}
The first positive zero of the derivative of $J_1$ is denoted by $j_1$.
\begin{thrm}\label{cor:main}
If for all $r>0$ and all $\theta\in\S^1$, $\hh(h(\ray \cos\angleS, \ray \sin\angleS)) = 
\sum_{k\in\sousz} c_k J_k(\mu \ray) e^{-ik\angleS}
$
for some map $\hh:\R^m\to\C$, $\mu>0$,
$(c_k)_{k\in\sousz}\in \C^{\sousz}$
and
$\sousz\subset\Z$ finite,
then
\eqref{E:system_depart0} is stabilizable over any compact set $\mathcal{K}$ in $B_{\R^2}(0, \frac{j_1}{\mu})$
by means of an infinite-dimensional piecewise constant dynamic output feedback.
\end{thrm}

The idea behind this result is that using representation theory, we are able to exhibit an embedding of System~\eqref{E:system_depart0} into a dissipative system over a Hilbert space. If the output of \eqref{E:system_depart0} can be transformed into a linear form of this space, this allows to write a classical Luenberger observer. 
This strategy is responsible for the particular form of $\hh\circ h$ in terms of Bessel functions: it corresponds precisely to the composition of the embedding with linear forms, as will be made explicit later.
In particular, this approach allows to recover output stabilization results with the output $h(x)=\frac{1}{2}|x|^2$ that was discussed in Section 3, as illustrated by the following corollary.

\begin{crllr}\label{cor:x2inf}
If $h(x)=\frac{1}{2}|x|^2$,
then \eqref{E:system_depart0} is semi-globally stabilizable
by means of an infinite-dimensional piecewise constant dynamic output feedback.
\end{crllr}

The proofs of Theorem~\ref{cor:main} and Corollary~\ref{cor:x2inf} are discussed in Section~\ref{sec:example_infinie} as an application of Theorem~\ref{th:infinite}.
As for the case $h(x) = \frac{1}{2}|x|^2$ that was treated  via a finite-dimensional strategy in Section~\ref{sec:finite},
one can devise a similar strategy for radially symmetric output maps where $\frac{1}{2}|x|^2$ is extracted at least locally around the target by inversion.
However, this is impossible if the output is not radially symmetric.
For instance,
if $h(\ray \cos\angleS, \ray \sin\angleS) = J_2(\mu r)\cos(2\angleS)$
for some $\mu>0$,
then the output is not radially symmetric and the system is unobservable at the target point. Still, according to Theorem~\ref{cor:main},
system~\eqref{E:system_depart0}
is stabilizable over any compact in
$B_{\R^2}(0, \frac{j_1}{\mu})$
by means of an infinite-dimensional embedding-based dynamic output feedback.

\subsection{Embedding into unitary systems and observer design}

Let $(\XX, \|\cdot\|_\XX)$ be a separable Hilbert space and $\dom$ be a dense subspace of $\XX$.
For all $u\in\R^p$, let $\opa(u) : \doma \to \XX$ be the skew-adjoint generator of a strongly continuous unitary group on $\XX$ and $\opc\in \lin(\XX, \C^{\mm})$ for some positive integer $\mm$.
Let $u:\R_+\to\R^p$ be piecewise constant, \ie such that
$u(t) = u(t_k)$ for all $t\in[t_k, t_{k+1})$ for some sequence $(t_k)_{k\in\N}$ in $\R_+$ with constant positive increment $t_{k+1} - t_k = \dt$.
Let $z_0\in\XX$.
Consider the non-autonomous linear abstract Cauchy problem with linear measured output
\begin{equation}\label{E:system_plonge_infinie}
\begin{aligned}
\begin{cases}
\dot \etat = \opa(u(t)) \etat\\
\etat(0) = \etat_0
\end{cases}
\end{aligned}
\qquad
\yy = \mC z.
\end{equation}
Acording to \cite[Chapter 5, Theorem 4.8]{Pazy}, the family $(\opa(u(t)))_{t\in\R_+}$ is the generator of a unique evolution system on $\XX$ that we denote by $(\evol_{t}(\cdot, u))_{t\in\R_+}$.
For any $\etat_0\in\XX$,~\eqref{E:system_plonge_infinie} admits a unique solution $\etat\in C^0(\R_+, \XX)$ given by $\etat(t) = \evol_{t}(\etat_0, u)$ for all $t \in\R_+$.
Moreover, if $\etat_0\in\doma$, then $\etat\in C^0(\R_+, \doma)$ and
is continuously differentiable (with values in $\XX$) on $[t_k, t_{k+1}]$ for all $k\in\N$.
The reader may refer to \cite[Chapter~5]{Pazy}, \cite[Chapter~VI.9]{engel2001one} or \cite{ito} for more details on the evolution equations theory.

For such systems, a Luenberger observer with constant gain $\rr>0$ can be built as follows:
\begin{equation}
\begin{aligned}
\begin{cases}
\dot \etath = \opa(u(t)) \etath - \rr\opc^*(\opc\etath-\yy)\\
\etath(0) = \etath_0\in\XX.
\end{cases}
\end{aligned}
\label{obs}
\end{equation}
Set $\eps = \etath-\etat$ and $\eps_0 = \etath_0-\etat_0$.
From now on, $\etath$ represents the state estimation made by the observer system and $\eps$ the error between this estimation and the actual state of the system.
Then $\etath$ satisfies~\eqref{obs} if and only if $\eps$ satisfies
\begin{equation}
\begin{aligned}
\begin{cases}
\dot \eps = (\opa(u(t)) - \rr\opc^*\opc)\eps\\
\eps(0) = \eps_0.
\end{cases}
\end{aligned}
\label{eps}
\end{equation}
Since $\opc\in\linxy$, \cite[Chapter 5, Theorem 2.3]{Pazy} claims that $(\opa(u(t))-\rr\opc^*\opc)_{t\ge 0}$ is also a stable family of generators of strongly continuous semigroups, and generates an evolution system on $\XX$ denoted by $(\evoleps_t(\cdot, u))_{t\in\R_+}$.
Then,
systems \eqref{obs} and~\eqref{eps} have respectively a unique solution $\etath$ and $\eps$ in $C^0(\R_+, \XX)$. 
Moreover, $\etath(t) =\evol_t(\etat_0, u)+\evoleps_t(\eps_0, u)$ and $\eps(t) = \evoleps_t(\eps_0, u)$ for all $t \in\R_+$. 
If $(\etath_0, \eps_0)\in\doma^2$,
$\etath$, $\eps\in C^0(\R_+, \doma)$ and
are continuously differentiable (with values in $\XX$) on $[t_k, t_{k+1}]$ for all $k\in\N$.

This infinite-dimensional Luenberger observer has been investigated in \cite{Celle-etal.1989} (see also \cite{brivadis:hal-02529820}), in which it is proved that $\eps(t)\cvf0$ as $t$ goes to infinity if $u$ is a \emph{regularly persistent input}.
Our goal is to embed the original system \eqref{E:system_general} into a unitary system, and to use this observer design in the context of dynamic output feedback stabilization.
In Section~\ref{sec:rep}, we exhibit an explicit embedding of \eqref{E:system_depart0} into \eqref{E:system_plonge_infinie}.

\begin{dfntn}[Embedding]
An injective map $\plong:\R^n\mapsto\XX$ is said to be an embedding\footnote{
This definition does not coincides with the usual notion of embedding in differential topology.}
of \eqref{E:system_general} into the unitary system \eqref{E:system_plonge_infinie} if
there exists $\hh:\R^m\to\C^\mm$ such that
the following diagram is commutative for all $t\in\R_+$ 
and any piecewise constant input $u:\R_+\to \R^p$:
\begin{equation}\label{diagram}
\xymatrix{
\R^n \ar[d]_\plong \ar[r]^{\varphi_t(\cdot, u)}
& \R^n \ar[d]^\plong \ar[r]^h
& \R^m \ar[r]^\hh
& \C^\mm
\\
\XX \ar[r]_{\sg_t(\plong(\cdot), u)}
& \XX \ar[urr]_\mC
}
\end{equation}
\emph{i.e.}, for all $x_0\in\R^n$,
$\plong(\varphi_t(x_0, u)) = \sg_t(\plong(x_0), u)$
and
$\hh(h(x_0)) = \mC \plong(x_0)$.
\end{dfntn}

Here, the map $\hh$ is a degree of freedom that may be chosen 
to find an embedding of \eqref{E:system_general} into~\eqref{E:system_plonge_infinie}.
Let $u:\R_+\to \R^p$ be piecewise constant,
$\etat_0,\, \eps_0\in\dom$, $\etat(t) = \evol_t(\zhat_0, u)$ and $\eps(t) = \evoleps_t(\eps_0, u)$ for all $t\in\R_+$.
For all $k\in\N$, $\mA(u(t_k))$ is skew-adjoint, hence for all $t\in[t_k, t_{k+1})$,
\begin{align}
    &\frac12\frac{\diff \norm{\etat}^2}{\diff t}(t)
    = \Re\psX{\mA(u(t))\etat(t)}{\etat(t)}
    = 0,
    \label{E:etat_constant}
    \\
    &\frac{1}{2}\frac{\diff \norm{\varepsilon}^2}{\diff t}(t)
    =\Re\psX{\mA(u(t))\varepsilon(t)}{\eps(t)} -\alpha\Re\psX{\mC^*\mC\varepsilon(t)}{\varepsilon(t)}
    = -\alpha \norm{\mC\varepsilon(t)}^2
    \leq 0.
    \label{E:eps_decroit_infinie}
\end{align}
Thus $\norm{\etat}$ is constant and $\norm{\varepsilon}$ is non-increasing.
If there exists a positive constant $\beta$ such that for all $x\in\dom$ and all $u\in\R$,
\begin{equation}\label{plong:eqCCA}
    \norm{\mC^*\mC\mA(u)x} \leq \beta \norm{x},
\end{equation}
then
\begin{align*}
\frac12\frac{\diff}{\diff t}\|\mC\eps(t)\|_{\C^\mm}^2
    &=  \langle\mC\eps(t), \mC\dot\eps(t)\rangle_{\C^\mm}\\
    &=  \langle\mC\eps(t), \mC\mA(u(t))\eps(t)\rangle_{\C^\mm} - \alpha \langle\mC\eps(t), \mC\mC^*\mC\eps(t)\rangle_{\C^\mm}\\
    &= \psX{\eps(t)}{\mC^*\mC\mA(u(t))\eps(t)} - \alpha \|\mC^*\mC\eps(t)\|_Z^2\\
    &\leq \beta \|\eps(0)\|_Z^2
\end{align*}
since $\norm{\varepsilon}$ is non-increasing.
Then,
$t\mapsto\norm{\mC\varepsilon(t)}^2$ is non-negative, integrable over $\R_+$ (by \eqref{E:eps_decroit_infinie}), and has bounded derivative (since $\mA(u(t))$ is skew-adjoint).
Hence, according to Barbalat's lemma,
$\mC\varepsilon(t)\to0$ as $t\cv+\infty$.
Inequality~\eqref{E:eps_decroit_infinie} is similar to \eqref{E:eps_decroit}, and will be a key argument to achieve the dynamic output feedback stabilization.

\subsection{
Embedding inversion:
from the embedded system's weak observer to the original system's observer
}

In Section~\ref{sec:finite}, a crucial argument was the existence of a left-inverse $\inv$ to the embedding $\plong$
(see \eqref{def:uppi}).
Now, $\XX$ being infinite-dimensional, we 
need to 
precise the notion of left-inverse, and, moreover, %
the convergence of the observer $\etath$ to the embedded state $\etat$ 
will hold only in the weak topology of $\XX$, namely, $\eps\cvf0$.
This is an important issue, which causes difficulties in achieving output feedback stabilization.
However, in this section, we show that if the original state $x$ remains bounded, and if the embedding $\plong$ is injective and analytic, then $\xhat=\inv(\etath)$ is actually an observer of $x$ in the usual topology of $\R^n$, namely, $\xhat-x\to0$.
This is summarized in Corollary~\ref{cor:conv_faible_forte}, which is an important result of the paper.

\begin{dfntn}[Strong left-inverse]
Let $(\XX, \|\cdot\|_\XX)$ be a normed vector space, $\Kx\subset\R^n$ and $\plong:\R^n\to\XX$.
A map $\inv:\XX\to\Kx$ is called a \emph{strong left-inverse} of $\plong$ on $\Kx$ if and only if there exists
a class $\Kinf$ function\footnote{A class $\Kinf$ function is a continuous function $\rho^*:\R_+\to\R_+$ such that $\rho^*(0)=0$, $\rho^*$ is strictly increasing and tends to infinity at infinity.} $\rho^*$ and $\qform\in\lin(\XX,\C^q)$ for some a positive integer $q$
such that,
for all $(x,  \xi)\in\Kx\times\XX$,
\begin{equation}\label{E:ps_inv}
    |\inv(\xi)-x|
    \leq \rho^*(|\qform(\xi-\plong(x))|).
\end{equation}
\end{dfntn}

\begin{rmrk}
If $\inv$ is a strong left-inverse of $\plong$ on $\Kx$,
then \eqref{E:ps_inv} implies that $\inv$ is also a left-inverse in the usual sense:
for all $x\in\Kx$, $\inv(\plong(x))=x$.
In particular, $\plong$ is injective over $\Kx$.
\end{rmrk}

The reason for which we look for a strong left-inverse of $\plong$ is the following lemma, which follows directly from \eqref{E:ps_inv} and the fact that $\qform\in\lin(\XX,\C^q)$.

\begin{lmm}\label{lem:ps_inv}
Let $(\XX, \|\cdot\|_\XX)$ be a normed vector space, $\Kx\subset\R^n$ and $\plong:\R^n\to\XX$.
Let $\inv:\XX\to\Kx$ be a strong left-inverse of $\plong$ on $\Kx$.
Let $(x_n)_{n\in\N}$ and $(\xi_n)_{n\in\N}$ be two sequences in $\Kx$ and $\XX$, respectively.
If $\xi_n - \plong(x_n) \cvf 0$ as $n$ goes to infinity,
then $\abs{\inv(\xi_n) - x_n}\cv0$ as $n$ goes to infinity.
\end{lmm}

This justifies the denomination of \emph{strong} left-inverse, in the sense that it allows to pass from weak convergence in the infinite-dimensional space $\XX$ to (usual) convergence in the finite-dimensional space $\R^n$.
The following theorem states sufficient conditions for the existence of a strong left-inverse.

\begin{thrm}\label{th:ps_inv}
Let $\XX$ be a separable Hilbert space, $\plong:\R^n\to\XX$ be an analytic map and $\Kx\subset\R^n$ be a compact set .
If $\plong|_\Kx$ is injective,
then $\plong$ has a continuous strong left-inverse on $\Kx$.
\end{thrm}
\begin{proof}
Let $(e_k)_{k\in\N}$ be a Hilbert basis of $\XX$.
For all $i\in\N$, let
$$E_i = \{(x_a, x_b) \in \R^n\times\R^n\mid \forall k\in\{0,\dots,i-1\},\ \psX{\plong(x_a)-\plong(x_b)}{e_k} = 0\}.$$
Then $(E_i)_{i\in\N}$ is a non-increasing family of analytic sets.
According to \cite[Chapter 5, Corollary 1]{Narasimhan}, $(E_i \cap \Kx^2)_{i\in\N}$ is stationary, \emph{i.e.}, there exists $q\in\N$ such that $E_q \cap \Kx^2 = E_i \cap \Kx^2$ for all $i\geq q$. Hence,
\begin{align}
E_q \cap \Kx^2
&= \bigcap_{k\in\N} E_k\cap \Kx^2\nonumber\\
&= \{(x_a, x_b) \in \Kx^2\mid \plong(x_a)=\plong(x_b)\} \tag{since $(e_k)_{k\in\N}$ is a Hilbert basis of $\XX$}\\
&= \{(x_a, x_a) \mid x_a\in \Kx\}\nonumber
\end{align}
since $\plong$ is injective on $\Kx$.
Let $\qform : \XX \ni \xi \mapsto (\psX{\xi}{e_k})_{k\in\{0,\dots,q-1\}} \in \C^q$ and $\tilde{\plong} = \qform \circ \plong$.
Then $\tilde{\plong}$ is continuous and injective on $\Kx$. Indeed, for all $(x_a, x_b)\in\Kx^2$, if $\tilde{\plong}(x_a) = \tilde{\plong}(x_b)$, then $(x_a, x_b)\in E_q \cap \Kx^2$ which yields $x_a=x_b$.
Hence, combining \cite[Lemma 6]{Bernard-etal.2017} and \cite[Theorem 1]{AndrieuPraly2006}, there exists a continuous map $\tilde{\inv}:\C^q\to\Kx$ and a class $\Kinf$ function $\rho^*$ such that for all $(x,  \mathfrak{z})\in\Kx\times\C^q$,
$|\tilde{\inv}(\mathfrak{z})-x| \leq \rho^*(|\mathfrak{z}-\tilde{\plong}(x)|).
$
Set $\inv = \tilde{\inv}\circ\qform$.
Then $\inv$ is continuous and for all $(x,  \xi)\in\Kx\times\XX$,
\[
    |\inv(\xi)-x|
    \leq \rho^*(|\qform(\xi)-\tilde{\plong}(x)|)
    = \rho^*(|\qform(\xi-\plong(x))|).
\]
\end{proof}
Applying Theorem~\ref{th:ps_inv}, then Lemma~\ref{lem:ps_inv}, we get the following important corollary in our context.

\begin{crllr}\label{cor:conv_faible_forte}
Let $\plong:\R^n\to\XX$ be an analytic embedding of \eqref{E:system_general} into the unitary system \eqref{E:system_plonge_infinie}
and $\Kx$ be a compact subset of $\R^n$.
Then $\plong$ has a continuous strong left-inverse $\inv$ on $\Kx$.

Let $x_0\in\Kx$, $\zhat_0\in\XX$ and
$u:\R_+\to\R^p$ piecewise constant.
Denote by $x$ and $\zhat$ the corresponding solutions of \eqref{E:system_general} and \eqref{obs}, respectively.
Set $\etat=\plong(x)$ and $\xhat = \inv(\zhat)$.
Assume that $x(t)\in\Kx$ for all $t\in\R_+$.
If $\zhat - \etat\cvf0$,
then $\xhat - x\cv0$.
\end{crllr}

\begin{rmrk}
Beyond the problem of output feedback stabilization, Corollary~\ref{cor:conv_faible_forte} may be used in the context of observer design.
In \cite{Celle-etal.1989}, after embedding the original finite-dimensional system into an infinite-dimensional unitary system, the authors investigate only the convergence of the infinite-dimensional observer.
Corollary~\ref{cor:conv_faible_forte} states that if the infinite-dimensional observer converges and if the original system's state trajectory remains bounded, then an observer can  be built for the original system, by using a strong left-inverse of the embedding.
\end{rmrk}

\subsection{Feedback perturbation and closed-loop system}

In order to set up a separation principle to solve the dynamic output feedback stabilization problem of \eqref{E:system_general}, let us assume that Condition~\ref{hyp:state_feedback} (semi-global) 
and that \eqref{E:system_general} admits an analytic embedding into the unitary system \eqref{E:system_plonge_infinie}.
Let $\Kx$ be a compact subset of $\R^n$.
Denote by $\phi$ a locally asymptotically stabilizing state feedback of \eqref{E:system_general} with basin of attraction containing $\Kx$ and by $\plong$ an embedding of \eqref{E:system_general} into \eqref{E:system_plonge_infinie}.
According to Theorem~\ref{th:ps_inv}, there exists $\inv:\XX\to\Kx$, a strong left-inverse of $\plong$ on $\Kx$.
Then, a natural way to build a dynamic output feedback would be
to combine \eqref{E:system_general}-\eqref{obs} with the control input $u = \phi(\inv(\etath))$, and to ensure that the state $x$ of \eqref{E:system_general} remains in $\Kx$.
However, due to the unobservability of the original system at the target, we propose, as in Section~\ref{sec_FeedbackPert}, to add a perturbation to this feedback law.
In \cite{Celle-etal.1989}, the convergence of the error system \eqref{eps} to $0$, when it holds, is only in the weak topology of $\XX$.
Therefore, the perturbation added to the feedback law must be chosen to vanish when the observer state $\etath$ of \eqref{obs} tends towards $\plong(0)$ in the weak topology.
For this reason, let us define a weak norm on $\XX$.

\begin{dfntn}[Weak norm]
Let $(e_k)_{k\in\Z}$ be a Hilbert basis of $\XX$.
For all $\xi\in\XX$,
set $$\psn(\xi) = \sqrt{\sum_{k\in\Z} \frac{|\psX{\xi}{e_k}|^2}{k^2+1}}.$$
Then $\psn$ defines a norm, we call the \emph{weak norm},
on $\XX$.
\end{dfntn}
Note that $\psn$ is not equivalent to $\|\cdot\|_\XX$, but satisfies
\begin{equation}\label{E:def_nu}
    \psn(\cdot)\leq \nu\|\cdot\|_\XX
\quad \text{ with }\quad  \nu = \sqrt{\sum_{k\in\Z} \frac{1}{k^2+1}} <+\infty.
\end{equation}
Moreover, $\psn$ induces a metric on bounded sets of $\XX$ endowed with the weak topology. More precisely, for any bounded sequence $(\xi_n)_{n\in\N}$ in $\XX$,
$\psn(\xi_n) \cv 0$ as $n$ goes to infinity
if and only if 
$\xi_n\cvf 0$ as $n$ goes to infinity.
Now, for some positive constant $\delta$ to be fixed (small enough) later, we can add the perturbation $\etath\mapsto\delta \psn^2(\zhat-\plong(0))$ to the feedback law.
Finally, the following coupled system is an infinite-dimensional piecewise constant dynamic output feedback of system~\eqref{E:system_general}:
\begin{equation}\label{E:system_stab_infinie}
\left\{
\begin{aligned}
&\dot x = f(x, u)\\
&y = h(x)
\end{aligned}
\right.
,\qquad
\left\{
\begin{aligned}
&\dot{\zhat}= \mA(u) \zhat -\rr\mC^*(\mC\zhat-\hh(\mes))\\
&u(t_k) = \phi(\inv(\zhat(t_k^-))) + \delta\psn^2(\zhat(t_k^-)-\plong(0))\\
&u(t)=u(t_k),\qquad t\in[t_k,t_{k+1})
\end{aligned}
\right.
\end{equation}
where $t_k = k\dt$ for all $k\in\N$ for some $\dt>0$.

\section{Revisiting the illustrative example}
\label{sec:example_infinie}

In this section, we illustrate the use of infinite-dimensional embeddings in the context of output feedback stabilization on a two-dimensional example with linear dynamics and nonlinear observation map.
Let  $h:\R^2\to\R^m$.
We consider the problem of stabilization by means of an infinite-dimensional embedding-based dynamic output feedback of the following system:
\begin{equation}\label{E:system_depart}
\left\{
\begin{aligned}
&\xdot= \matR x + b u
\\
&y= h(x)
\end{aligned}
\right.
\qquad
\text{with}\
\matR = \begin{pmatrix}
0&-1
\\
1&0
\end{pmatrix}
\ \text{and}\
b = \begin{pmatrix}
0
\\
1
\end{pmatrix}.
\end{equation}

Since $(A, b)$ is stabilizable, there exists $K\in\R^{1\times2}$ such that $A+bK$ is Hurwitz.
Moreover, $A$ is skew-symmetric. Hence $\kappa = |K|$ can be chosen arbitrarily small.
Then, the state feedback law $\phi:x\mapsto Kx$
is such that \eqref{E:system_depart} with $u = \phi(x)$ is globally asymptotically stable at 0.
Note that \eqref{E:system_depart} does not exactly fit the form of~\eqref{E:system_rot} since $h$ is not necessarily radially-symmetric.
Of course, our analysis is of interest only if~\eqref{E:system_depart} is not uniformly observable. In Example~\ref{ex:Jk}, we exhibit a non-radially symmetric $h$ that makes the system non-uniformly observable, and on which our (infinite-dimensional) embedding-based strategy does apply.
In the following we give some sufficient conditions on $h$ allowing the design of a stabilizing infinite-dimensional dynamic output feedback.
The main result of this section, Theorem~\ref{th:infinite} (stated in Section~\ref{sec:obs_infinie}), relies on three main hypotheses: the existence of an embedding of~\eqref{E:system_depart} 
into~\eqref{E:system_plonge_infinie}, 
and two observability assumptions.
For each of these assumptions, we provide examples of output maps $h$ satisfying theses hypotheses.

We wish to underline that in the rest of the paper, all mentions of prior occurrences of $f(x,u)$ make the assumption that $f(x,u)=Ax+bu$.

\subsection{Unitary representations and embeddings}
\label{sec:rep}

In \cite{Celle-etal.1989}, the authors investigated the problem of observer design for~\eqref{E:system_depart} by means of infinite-dimensional embeddings.
We briefly recall their strategy, that relies on representation theory (see, \eg \cite{vilenkin1978special, barut}).
The Lie group $\GG$ of system~\eqref{E:system_depart} (the group of flows generated by the dynamical system \eqref{E:system_depart} with constant inputs)
is isomorphic to $\R^2 \rtimes_\RR H$, where $H\simeq \{e^{tA}, t\in\R_+\} \simeq \S^1$ is the group of rotations (isomorphic to the unit circle), $\RR:\S^1\ni\angleS\mapsto e^{\angleS A}$ is an automorphism of $\R^2$
and $\rtimes_\RR$ denotes the outer semi-direct product with respect to $\RR$.
Hence $\GG$ is the group of motions of the plane.
According to \cite[Section IV.2]{vilenkin1978special},
its unitary irreducible representations are given by a family  $(\rep_\mu)_{\mu >0}$, where for each $\mu>0$,
\fonction{\rep_\mu}{\GG}{\lin(L^2(\S^1, \C))}{(x, \vartheta)}{
\left(  \xi\in L^2(\S^1, \C) \mapsto \left(\S^1\ni\varS\mapsto e^{i\mu (1, 0) e^{\varS A'} x}
\xi(\varS-\vartheta)\right)   \right).}
Let $\XX = L^2(\S^1, \C)$ be
the set of real-valued square-integrable functions over $\S^1$.
Then $\XX$ is a Hilbert space endowed with the scalar product defined by $\psX{\xi}{\zeta} = \frac{1}{2\pi}\int_0^{2\pi}\xi(\varS)\bar{\zeta}(\varS)\diff \varS$ and the induced norm $\|\cdot\|_\XX$.
Since $\S^1$ is compact, the constant function $\one:\varS\mapsto1$ lies in $\XX$.
Let $\mu>0$ to be fixed later.
Set
\fonction{\plong_\mu}{\R^2}{\XX}{x}{\rep_\mu(x, 0)\one.}
Note that $\plong_\mu$ depends on $\mu$, but from now on we omit this dependence in the notation and write $\plong$ instead of $\plong_\mu$.
Since $\rho_\mu$ is a unitary representation, $\|\plong(x)\|_\XX=1$ for all $x\in\R^2$ and $\plong(0) = \one$.
For all $x=(x_1, x_2)=(\ray\cos\angleS, \ray \sin\angleS)$ in $\R^2$, we have
\begin{equation}\label{E:deftau}
    \plong(x):\S^1\ni\varS\mapsto e^{i\mu(
    x_1\cos(\varS)+
    x_2\sin(\varS))
    }=
    e^{i\mu \ray\cos(\varS-\angleS)
    }.
\end{equation}
If $x, \tilde{x}\in\R^2$ are such that $\plong(x) = \plong(\tilde{x})$,
then $(x_1 - \tilde{x}_1)\cos(\varS) + (x_2 - \tilde{x}_2)\sin(\varS) = 0$ for all $\varS\in\S^1$, hence $x = \tilde{x}$.
Thus $\plong$ is injective.
Let $x$ be a solution of \eqref{E:system_depart} and set $\etat = \plong(x)$
Then
\begin{align*}
    \dot \etat
    &= i \mu \left(\xdot_1 \cos(\varS) + \xdot_2 \sin(\varS)\right) \etat\\
    &= i \mu \left(-x_2 \cos(\varS) + x_1 \sin(\varS)
    + u \sin(\varS) \right) \etat\\
    &= - \frac{\partial \etat}{\partial \varS} + i u \mu \sin(\varS) \etat\\
    &= \mA(u) \etat
\end{align*}
with $\mA(u) = -\frac{\partial}{\partial \varS} + i u \sin(\varS)$ defined on the dense domain $\dom = H^1(\S^1, \C) = \{f\in\XX\mid f'\in\XX\}$.
The operator $\mA(u)$ is the skew-adjoint generator of a strongly continuous unitary group on $\XX$ for any $u\in\R$.
In order to make $\plong$ an embedding of \eqref{E:system_depart} into \eqref{E:system_plonge_infinie}, we need the output map to be in the form $\mathfrak y = \mC z$. This is where the degree of freedom $\hh$ introduced in \eqref{E:system_stab_infinie} may be employed.
More specifically, we make the following  first assumption on the observation map $h$.
\begin{assumption}[Linearizable output map]\label{ass:h}
There exist $\hh:\R^m\to\C^\mm$ and $\mC\in \lin(\XX, \C^\mm)$ such that
$\hh(h(x)) = \mC\plong(x)$ for all $x\in\R^2$.
\end{assumption}

If Assumption~\ref{ass:h} holds for System~\eqref{E:system_depart}, then the closed loop infinite-dimensional piecewise constant dynamic output feedback system can take the form of \eqref{E:system_stab_infinie}.

For all $k\in\Z$, let
\fonction{e_k}{\S^1}{\C}{\varS}{e^{ik\varS}.}
The family $(e_k)_{k\in\Z}$ forms a Hilbert basis of $\XX$.
In the rest of the paper, the weak norm $\NN$ is always defined with respect to this Hilbert basis.
Then, for all $x=(\ray \cos\angleS, \ray \sin\angleS)\in\R^2$ and all $k\in\Z$,
\begin{align}
    \psX{\plong(x)}{e_k}
    &= \frac{1}{2\pi} \int_0^{2\pi} e^{i\mu\ray\cos(\varS-\angleS)-ik\varS}\diff \varS\nonumber\\
    &= \frac{1}{2\pi} e^{-ik\angleS + i\frac{\pi}{2}}\int_0^{2\pi} e^{i\mu\ray\sin(\varS)-ik\varS}\diff \varS\nonumber\\
    &=i^k J_k(\mu \ray) e^{-ik\angleS} \label{eq_Bessel}
\end{align}
where $J_k$ denotes the Bessel function of the first kind of order $k\in\Z$ (see \eqref{Bessel_def}).
Since $(e_k)_{k\in\Z}$ is a Hilbert basis of $\XX$ we get the following interpretation of the linearizable output map assumption in the case of \eqref{E:system_depart}.

\begin{prpstn}[Outputs satisfying Assumption~\ref{ass:h}]\label{prop:Jk}
If there exist a map $\hh:\R^m\to\C$ and a sequence $(c_k)_{k\in\Z}\in l^2(\Z, \C)$ such that  %
$\hh(h(\ray \cos\angleS, \ray \sin\angleS)) = 
\sum_{k\in\Z} c_k J_k(\mu \ray) e^{-ik\angleS}
$
then \eqref{E:system_depart} satisfies Assumption~\ref{ass:h}.
\end{prpstn}

\begin{xmpl}\label{ex:Jk}
The three observation maps
$h(x) = J_0(\mu|x|)-1$ (with $\hh(y) = y+1$),
$h(x) = J_2(\mu|x|)\cos(2\angleS)$ (with $\hh(y) = y$)
and $h(x) = |x|$ (with $\hh(y) = J_0(\mu y)$)
are suitable.
In each of these cases, the constant input $u\equiv0$ makes \eqref{E:system_depart} unobservable.
Moreover,  $h(x) = J_0(\mu|x|)-1$ and $h(x)=|x|$ are radially symmetric but $h(x) = J_2(\mu|x|)\cos(2\angleS)$ is not.
If $h(x) = |x|$, then \eqref{E:system_depart} is a subcase of system~\eqref{E:system_rot}.
\end{xmpl}

\begin{rmrk}\label{rem:gr}
According to the Gelfand--Raïkov theorem,
the finite linear combinations of pure positive-type functions
(functions of the form $(x, \vartheta)\mapsto\psX{\rep_\mu(x, \vartheta)\xi}{\xi}$, where $\mu>0$ and $\xi\in\XX$)
is dense for the uniform convergence on compact sets, in the continuous bounded
complex-valued functions on $\GG$.
Hence, the set of functions
$(\ray \cos\angleS, \ray \sin\angleS) \mapsto
\sum_{\ell\in I_1} \sum_{k\in I_2} c_k J_k(\mu_\ell \ray) e^{-ik\angleS}
$, where $I_1$ and $I_2$ are finite subsets of $\Z$,
$\mu_\ell>0$ and $c_k\in\C$,
is dense for the uniform convergence on compact sets of $\R^2$, in the continuous bounded
complex-valued functions on $\R^2$.
In the examples of applications of our results,
we will focus on output maps $h$ of the form
$\hh(h(x)) = \sum_{k\in I} c_k J_k(\mu \ray) e^{-ik\angleS}$
for some $\hh:\R^m\to\C^\mm$ and some fixed $\mu>0$.
\end{rmrk}

\subsection{Explicit strong left-inverse}

Having in mind to use the strategy developed in the previous section, we now explicitly construct a strong left-inverse $\inv$ of $\plong$ defined in \eqref{E:deftau} over some compact set.
With Corollary \ref{cor:conv_faible_forte}, we already know that a strong left-inverse $\inv$ exists. However, we would like to give an explicit expression. This can be done by employing the relationship between Bessel functions of the first kind given in \eqref{Bessel_def} and the embedding $\plong$, as shown in equation \eqref{eq_Bessel}.

Indeed, let $j_1$ denote the first zero of $J_1'$.
Then $J_1$ is increasing over $[0, j_1]$. Denote $J_1^{-1}$ its inverse over $[0, j_1]$.
Let $\Phi:\C\ni x_1+ix_2\mapsto(x_1, x_2)\in\R^2$ be the canonical bijection.
Let $\jj\in(0, j_1)$.
For all $\zeta\in\C$, let
\begin{equation}\label{E:def_f}
    \invf(\zeta) = \begin{cases}
    0
    & \text{if } \zeta=0\\
    \Phi\left(\frac{i\bar{\zeta}}{\mu |\zeta|} J_1^{-1}(|\zeta|)\right) 
    & \text{if } 0<|\zeta|\leq J_1(\jj)\\
    \Phi\left(\frac{i\bar{\zeta}}{\mu |\zeta|} j_1\right)
    & \text{if } |\zeta|\geq J_1(j_1)
    \end{cases}
\end{equation}
If $J_1(\jj)< |\zeta| < J_1(j_1)$,
define $\invf(\zeta)$ such that $\invf$ is continuously differentiable and globally Lipschitz over $\C$.
Denote by $\Lf$ its Lipschitz constant.
Let $e_1\in \XX$ be defined by $e_1(\varS) = e^{i\varS}$ for all $\varS\in\S^1$.
Let
\begin{equation}\label{E:definv}
\displaystyle
\begin{array}{lrcl}
\inv: & \XX & \longrightarrow & \R^2 \\
    & \xi & \longmapsto & \invf(\psX{\xi}{e_1})
\end{array}
\end{equation}

\begin{lmm}\label{lem:explicit_pi}
The map $\inv$ is a strong left-inverse of $\plong$ over $\bar{B}_{\R^2}(0, \frac{\jj}{\mu})$.
\end{lmm}
\begin{proof}
Set $\Kx = \bar{B}_{\R^2}(0, \frac{\jj}{\mu})$.
According to \eqref{E:def_f}, $\phi(\xi)\in\Kx$ for all $x\in\Kx$.
Let $x=(\ray \cos\angleS, \ray \sin\angleS)$ in $\Kx$.
Then, with \eqref{eq_Bessel},
\begin{align*}
    \psX{\plong(x)}{e_1}
    = ie^{-i\angleS}J_1(\mu \ray)
    \in\B_{\C}\left(0, J_1\left(\jj\right)\right).
\end{align*}
Hence $\inv(\plong(x))=\Phi(re^{i\angleS})=x$.
Let $\xi\in\XX$.
We have
\begin{align*}
    |\inv(\xi)-x|
    = |\inv(\xi)-\inv(\plong(x))|
    = |\invf(\psX{\xi}{e_1})-\invf(\psX{\plong(x)}{e_1})|
    \leq \Lf |\psX{\xi - \plong(x)}{e_1}|.
\end{align*}
Hence $\inv$ is a strong left-inverse of $\plong$ over $\Kx$.
\end{proof}

\begin{rmrk}
Letting $\mu$ tend towards $0$, the domain of the left-inverse tends towards $\R^2$. This will be of use to achieve semi-global stabilization.
\end{rmrk}

\subsection{Well-posedness and boundedness of trajectories}

We now check the well-posedness of the closed-loop system \eqref{E:system_stab_infinie}.
In a second step, since $\inv(\xi)$ is meaningful only if $|\psX{\xi}{e_1}|\leq J_1(\jj)$, we show that by  selecting the (perturbation) parameter $\delta$ sufficiently small, $\etath$ remains in this domain  along the  trajectories of the closed-loop system. 

\begin{lmm}\label{lem:well}
For all $\mu, \alpha, \delta
,\dt
>0$
and all $x_0, \xhat_0$ in $\bar{B}_{\R^2}(0, \frac{\jj}{\mu})$,
the system~\eqref{E:system_stab_infinie}
(with $\inv$ as in Lemma~\ref{lem:explicit_pi}) admits a unique solution
$(x, \etath)\in C^0(\R_+, \R^2\times \D)$
such that
$x(0) = x_0$ and $\etath(0) = \plong(\xhat_0)$.
Moreover, for all $k\in\N$,
$(x, \etath)\vert_{[t_k, t_{k+1}]}\in C^1([t_k, t_{k+1}], \R^2\times \XX)$.
\end{lmm}

\begin{proof}
Let $\Kx = \bar{B}_{\R^2}(0, \frac{\jj}{\mu})$ and $x_0$, $\xhat_0$ in $\Kx$.
Set $\etat_0 = \plong(x_0)\in\dom$ and $\eps_0 = \plong(\xhat_0) - \plong(x_0)\in\dom$.
The well-posedness of system~\eqref{E:system_stab_infinie}
is equivalent to the well-posedness of the following system:
\begin{equation}\label{E:syst_hyp}
\left\{
\begin{aligned}
&\dot{\etat}= \mA(u)\etat
\\
&\dot{\eps}= (\mA(u) - \alpha\mC^*\mC)\eps
\\
&u(t_k) = \phi(\inv(\etat(t_k^-)+\eps(t_k^-))) + \delta\psn^2(\etat(t_k^-)+\eps(t_k^-)-\one)\\
&u(t) = u(t_k),\qquad t\in[t_k,t_{k+1})
\\
&\etat(0)=\etat_0,\ \eps(0)=\eps_0
\end{aligned}
\right.
\end{equation}
where $\mA(u) = -\frac{\partial}{\partial \varS} + i\mu u \sin$ and $\mC\in \lin(\XX, \C^\mm)$.
For all $u\in\R$, $\mA(u)$ is the generator of a strongly continuous semigroup on $\XX$.
Since, $\mC$ is bounded, according to \cite[Chapter 3, Theorem 1.1]{Pazy}, $\mA(u)-\alpha \mC^*\mC$ is also the generator of a strongly continuous semigroup on $\XX$.
Thus, reasoning by induction on $k\in\N$,
\eqref{E:syst_hyp} admits a unique solution $(\etat, \eps)\in C^0(\R_+, \XX^2)$.
Moreover, since $z_0, \eps_0\in\doma$, $(\etat, \eps)\in C^0(\R_+, \doma^2)$ and
is continuously differentiable (with values in $\XX^2$) on $[t_k, t_{k+1}]$ for all $k\in\N$.
Setting $x=\inv(\etat)$ and $\zhat=z+\eps$, we get the statement.
\end{proof}

Now that existence and uniqueness of solutions of \eqref{E:system_stab_infinie} have been proved,
let us show that the trajectories are bounded.

\begin{lmm}\label{lem:borne}
For all $\mu>0$, all $\Rc\in(0, \frac{\jj}{\mu})$ and all $\Rb\in(0, \Rc)$, there exist $\Ra\in(0, \Rb)$,
$\delta_0>0$ 
and $\dt_0>0$
such that 
for all $x_0, \xhat_0$ in $\B_{\R^2}(0, \Ra)$,
all $\alpha>0$,
all $\delta\in(0, \delta_0)$
and all $\dt\in(0, \dt_0)$,
the unique solution
$(x, \etath)\in C^0(\R_+, \R^2\times\doma)$
of \eqref{E:system_stab_infinie}
such that
$x(0) = x_0$ and $\etath(0) = \plong(\xhat_0)$
satisfies
$|x(t)| < \Rb$,
$|\psX{\etath(t)}{e_1}| < J_1(\mu\Rc)$
and $|\inv(\etath(t))| < \Rc$
for all $t\in\R_+$.
\end{lmm}
\begin{proof}
For any bounded $u:\R^+\to \R$, we can decompose the dynamics of $x$ as 
$\dot x = (\rot+bK)x + b(u - Kx)$. Thus for any initial condition $x_0\in \R^2$, the variation of constants formula yields
$$
x(t)=\e^{t(\rot+bK)}x_0+\int_0^t\e^{(t-s)(\rot+bK)}b(u(s) - Kx(s))\diff s.
$$
Since $\rot+bK$ is Hurwitz, $\|\e^{t(\rot+bK)}\|\leq 1$ and there exists  constants $\gamma_1,\gamma_2>0$ such that $\left|\e^{t(\rot+bK)}b\right|\leq \gamma_1\e^{-\gamma_2 t}$ for all $t\geq 0$. Then there exists $M>0$ (depending only on $A$, $b$ and $K$) such that 
\begin{equation}\label{E:varcons}
|x(t)| \leq |x_0| + M\sup_{s\in[0, t]} |u(s) - Kx(s)|\qquad \forall t\in\R_+.
\end{equation}
We use this expression to bound $|x|$.
Recall $\kappa = |K|$.
Denote by $\Linv$ the global Lipschitz constant of $\inv$.
Since $J_0(0)=1$, we can pick  $\Ra \in(0, \Rb)$, $\delta_0>0$
and $\dt_0>0$
small enough so that the following inequalities are satisfied (with $\nu$ defined by \eqref{E:def_nu}):
\begin{align}
&
\Ra + M\left(
2\kappa \Linv\sqrt{2(1-J_0(\mu R_0))} + 16\nu^2\delta_0
+ \kappa \dt_0 (\Rb + 3\kappa\Linv + 16\nu^2\delta_0)
\right)
< \Rb,
\label{Ineq1}
\\
&2\sqrt{2(1-J_0(\mu \Ra))} +  J_1(\mu \Rb) <  J_1(\mu \Rc).
\label{Ineq2}
 \end{align}

Let $\delta\in(0, \delta_0)$,
$\dt\in(0, \dt_0)$,
$x_0, \xhat_0\in\B_{\R^2}(0, \Ra)$,
$(x, \etath)$ as in Lemma~\ref{lem:well},
$\etat = \plong(x)$, $\eps = \etath-\etat$,
$t_k=k\dt$ for $k\in\N$,
$u(t_k) = \phi(\inv(\zhat(t_k^-))) + \delta\psn^2(\zhat(t_k^-)-\one)$
and $u(t) = u(t_k)$ for $t\in[t_k, t_{k+1})$.
Let $k\in\N$ and $t\in[t_k, t_{k+1})$.
The expression of $u$ allows to split $|u- Kx|$ into  three parts (recall $\inv(z(t_k))=x(t_k)$):
\begin{align}\label{E:control_maj}
    |  u(t) - Kx(t) |
    \leq \kappa |\inv(\etath(t_k)) - \inv(\etat(t_k))|
    + \kappa|x(t_k) - x(t)| + \delta \psn^2(\etath(t_k)-\one).
\end{align}
We bound these terms going right to left. First,
\begin{align}
    \psn^2(\etath(t_k)-\one)
    &\leq\nu^2\norm{\etath(t_k)-\one}^2
    \tag{by \eqref{E:def_nu}}\\
    &\leq\nu^2\left(\norm{\eps(t_k)} + \norm{\etat(t_k)-\one}\right)^2\tag{by triangular inequality}\\
    &\leq\nu^2\left(\norm{\eps_0} + 2\right)^2\tag{since $\norm{\eps}$ is non-increasing and $\norm{\plong(x(t))}=1$}
\end{align}
Since $\norm{\eps_0}\leq\norm{\hat z_0}+\norm{z_0}=1+1$, this yields 
\begin{equation}\label{E:term_right}
   \psn^2(\etath(t_k)-\one) \leq 16\nu^2.
\end{equation}
This allows to give a bound on $u(t_k)$:
$$
   |u(t_k)|
   \leq \kappa|\inv(\zhat(t_k))|+\delta\psn^2(\etath(t_k)-\one)
   \leq \kappa\Linv(\norm{\eps(t_k)}+\norm{z(t_k)})+ 16\nu^2\delta\leq 3\kappa\Linv + 16\nu^2\delta.
$$
Then, with another variation of constants, we obtain (since $|b|=1$, and $\Delta$ small enough)
\begin{equation}\label{E:term_middle}
    |x(t_k) - x(t)|\leq |(\e^{(t-t_k)A}-\Id_{\R^2})x(t_k)|+|u(t_k)|(t-t_k)
    \leq \dt \left(|x(t_k)| + 3\kappa\Linv + 16\nu^2\delta \right).
\end{equation}
Finally, note that for any $x_0\in B_{\R^2}(0,R_0)$,
$$
    \norm{\plong(x_0)-\one}
    = \left(\norm{\plong(x_0)}^2+1 - 2\psX{\plong(x_0)}{1}\right)^{\frac{1}{2}}
    = \sqrt{2(1-J_0(\mu |x_0|))}
    \leq \sqrt{2(1-J_0(\mu R_0))}.
$$
Then 
$$
    \norm{\eps_0} \leq \norm{\etath_0-\one} + \norm{\etat_0-\one}
    \leq 2\sqrt{2(1-J_0(\mu R_0))},
$$
which implies
\begin{equation}\label{E:term_left}
    |\inv(\etath(t_k)) - \inv(\etat(t_k))|
    \leq \Linv \norm{\eps(t_k)}
    \leq \Linv \norm{\eps_0}
    \leq 2\Linv\sqrt{2(1-J_0(\mu R_0))}. 
\end{equation}

In conclusion, $\kappa\times\eqref{E:term_left}+\kappa\times\eqref{E:term_middle}+\delta\times\eqref{E:term_right}$ implies that \eqref{E:control_maj} becomes
\begin{align}
    | u(t) - Kx(t)| &\leq 2\kappa \Linv\sqrt{2(1-J_0(\mu R_0))} 
    + \kappa \dt \left(|x(t_k)| + 3\kappa\Linv + 16\nu^2\delta\right)
    + 16\nu^2\delta
    \label{Ineq_e}
\end{align}
Assume for the sake of contradiction that there exists $t\in\R_+$ such that $|x(t)|>\Rb$.
Let $T = \min\{t\in\R_+\mid |x(t)|=\Rb\}$.
Then \eqref{E:varcons} at $t=T$ combined with \eqref{Ineq1} and \eqref{Ineq_e}  yields
\begin{align*}
\Rb \leq \Ra + M\left(
2\kappa \Linv\sqrt{2(1-J_0(\mu R_0))} 
+ 
\kappa \dt (\Rb + 3\kappa\Linv + 16\nu^2\delta)
+ 
16\nu^2\delta
\right)
< \Rb,
\end{align*}
which is a contradiction.
Thus
$|x(t)|<R_1$ for all $t\in\R_+$.
Furthermore, for all $t\in\R_+$,
\begin{align*}
    \left|\psX{\etath(t)}{e_1}\right| &\leq |\psX{\eps(t)}{e_1}| + |\psX{\etat(t)}{e_1}|\\
    &\leq \norm{\eps(t)} + |\psX{\plong(x(t))}{e_1}|\\
    &\leq \norm{\eps_0} +  J_1(\mu |x|)\\
    &\leq 2\sqrt{2(1-J_0(\mu \Ra))} + J_1(\mu \Rb)\\
    &<  J_1(\mu \Rc).
\end{align*}
Thus, \eqref{Ineq2} yields
$|\psX{\etath(t)}{e_1}| < J_1(\mu\Rc)$
for all $t\in\R_+$.
Finally, since $J_1(\mu\Rc) < J_1(\jj)$,
$|\inv(\etath(t))| = |\invf(\psX{\etath(t)}{e_1})|
= \frac{1}{\mu} J_1^{-1}(\psX{\etath(t)}{e_1})
\leq \Rc.
$
\end{proof}

For any compact set of initial conditions, taking 
$\mu$, $\delta$ and $\dt$
sufficiently small ensures that trajectories are bounded. This is shown in the following corollary.

\begin{crllr}\label{cor:semiglob}
For all $\Ra>0$, there exist
$\mu_0, \delta_0, \dt_0>0$ 
and $\Rc>\Rb>\Ra$ such that
for all $x_0, \xhat_0$ in $\B_{\R^2}(0, \Ra)$,
all $\mu\in(0, \mu_0)$, all $\delta\in(0, \delta_0)$, all $\dt\in(0, \dt_0)$,
and all $\alpha>0$,
the unique solution
$(x, \etath)\in C^0(\R_+, \R^2\times \D)$
of \eqref{E:system_stab_infinie}
such that
$x(0) = x_0$ and $\etath(0) = \plong(\xhat_0)$
satisfies
$|x(t)| < \Rb$,
$|\psX{\etath(t)}{e_1}| < J_1(\mu\Rc)$
and $|\inv(\etath(t))| < \Rc$
for all $t\in\R_+$.
\end{crllr}

\begin{proof}
Let $\beta_2> \beta_1 >1$ to be fixed later, and let $\Rb = \beta_1\Ra$ and $\Rc = \beta_2\Ra$.
Then there exist $\mu_0, \delta_0,\dt_0>0$ small enough such that
\eqref{Ineq1} holds 
for all $\mu\in(0, \mu_0)$,
$\delta\in(0, \delta_0)$ and $\dt\in(0, \dt_0)$,
Recall the following asymptotic expansions of the Bessel functions of the first kind at 0:
\begin{align*}
    J_0(r) = 1 - \frac{r^2}{4} + \oo(r^2),\qquad J_1(r) = \frac{r}{2} + \oo(r).
\end{align*}
Then for all $\mu>0$,
\begin{align*}
2\sqrt{2(1-J_0(\mu\Ra))} +  J_1(\mu \Rb)
= \mu\Ra \left(\sqrt{2} +  \frac{\beta_1}{2}\right) + \oo(\mu),\
J_1(\mu \Rc) = \mu\Ra\frac{\beta_2}{2} + \oo(\mu).
\end{align*}
Hence, if $\beta_2>2\sqrt{2}+\beta_1$, then there exists $\mu_0>0$ such that \eqref{Ineq2} holds for all $\mu\in(0, \mu_0)$.
Set $\beta_1 = 2$ and $\beta_2 = 2\sqrt{2}+3$.
Then there exist $\mu_0,
\delta_0, \dt_0>0
$
such that $\mu_0\Rc<\jj$ and \eqref{Ineq1} and \eqref{Ineq2} are satisfied for all 
$\mu\in(0, \mu_0)$,
$\delta\in(0, \delta_0)$ and $\dt\in(0, \dt_0)$.
Reasoning as in the proof of Lemma~\ref{lem:borne}, the result follows.
\end{proof}

\subsection{Observability analysis}\label{sec:obs_infinie}

To achieve state estimation in the embedded system in presence of an observability singularity at the target, we need to check that enough information is still accessible during the stabilization process. This takes the form of two last assumptions on the linear output map $\opc\in\linxy$ (obtained from the function $h$ in Assumption~\ref{ass:h}), reminding of the discussions in Section~\ref{sec:nc}, except in the new infinite-dimensional context.
The \emph{short time 0-detectability} assumption states that the measure distinguishes the target point for its neighbors.
The \emph{isolated observability singularity} assumption regards separation of the unobservable input $u\equiv0$  from other singular inputs of the infinite-dimensional system.
For each assumption, we discuss examples of suitable output maps $h$.
Following Remark~\ref{rem:gr}, we investigate the case where at least one of the components of the output map is in the linear span of a finite number of elements of the Hilbert basis.
This component is used to ensure the two observability properties.

\begin{assumption}[Short time 0-detectability]\label{ass:detec}
Let $u:\R_+\to\R$ be constant over $[t_k, t_{k+1})$ for all $k\in\N$, where $t_{k+1}-t_k = \dt$ is a positive constant.
Let $x$ be a solution of \eqref{E:system_depart}
bounded by $\frac{\jj}{\mu}$.
If there exists a subsequence $(t_{k_n})_{n\in\N}$ such that
$u(t_{k_n})\cvl{n\cv+\infty} 0$
and
$\opc\plong(x(t_{k_n}+t')) \cvl{n\cv+\infty} \opc\plong(0)$
for all $t'\in[0, \dt]$,
then
$x(t_{k_{n}}) \cvl{n\cv+\infty} 0$.
\end{assumption}

\begin{rmrk}
Assumption~\ref{ass:detec} implies the necessary Condition~\ref{hyp:distinguish}~(local).
Indeed, if $x$ is a solution of \eqref{E:system_depart} with $u=0$ and $h(x(t))=0$ for all $t\geq0$,
then for any positive increasing sequence $(t_n)_{n\in\N}\to+\infty$,
$u(t_n) = 0$ and $C\plong(x(t_n+t)) = \hh(0)$ for all $n\in\N$ and all $t\geq0$. Hence, according to Assumption~\ref{ass:detec}, $x(t_n)\to0$ and Condition~\ref{hyp:distinguish}~(local) is satisfied.
Moreover,
if $\hh$ has a continuous inverse in a neighborhood of $0$, then
Assumption~\ref{ass:detec} implies
(for piecewise constant inputs only)
the input/output-to-state stability
condition
(see \eg \cite{krichman2001input}), which states that any solution $x$ of \eqref{E:system_depart} such that $u(t)\to0$ and $y(t)\to0$ is such that $x(t)\to0$ as $t\to+\infty$. This condition has proved to be of interest in the context of output feedback stabilization.
\end{rmrk}

\begin{prpstn}[Outputs satisfying Assumption~\ref{ass:detec}]\label{prop:detec}

If one of the components of $\opc$ (seen as a $\mm$-tuple of linear forms on $\XX$) has the form $\psX{\cdot}{\zeta}$
where 
$\zeta = \sum_{p\in\sousz}c_pe_p\in\XX\setminus\{0\}$,
with $\sousz\subset\Z$ finite, then \eqref{E:system_depart} satisfies Assumption~\ref{ass:detec}.
\end{prpstn}
\begin{proof}
Since $x$ is bounded, we show that its only accumulation point is $0$.
With no loss of generality, we may assume that
$x(t_{k_n})$ tends towards some $x^\star = (r^\star\cos \theta^\star, r^\star\sin \theta^\star)\neq0$.
According to \eqref{eq_Bessel}, we have $\psX{\plong(x(t_{k_n}+t'))}{\zeta}=\sum_{p\in\sousz} c_p J_p(\mu \ray(t_{k_n}+t')) e^{-ip\angleS(t_{k_n}+t')}$, with $x=(\ray  \cos\angleS , \ray \sin\angleS )$.
As $n$ goes to $+\infty$, since $\opc\plong(x(t_{k_n}+t'))$ tends towards
$\opc\plong(0)$,
$\psX{\plong(x(t_{k_n}+t'))}{\zeta}$ 
tends towards $c_0$ if $0\in I$, or to $0$ otherwise.
In particular, for $t'=0$, we get that
$\sum_{p\in\sousz} c_p J_p(\mu \ray^\star) e^{-ip\angleS^\star} = c_0$
if $0\in I$, or $0$ otherwise.
Thus $c_0J_0(\mu r^\star)=c_0$ if $0\in I$, and $c_pJ_p(\mu \ray^\star)=0$ for $p\in I\setminus\{0\}$.
Since $u(t) = u(t_{k_n})$ for all $t\in[t_{k_n}, t_{k_n+1}]$ tends towards $0$ as $n$ goes to $+\infty$, Duhamel's formula implies that for all $t'\in[0, \dt]$,
\begin{equation*}
    x(t_{k_n}+t') - e^{t' A}x(t_{k_n}) \cvl{n\cv+\infty} 0,
\end{equation*}
\ie
\begin{equation*}
    r(t_{k_n}+t') - r(t_{k_n})\cvl{n\cv+\infty} 0
    \ \text{and}\
    e^{i\angleS(t_{k_n}+t')} - e^{i(\angleS(t_{k_n})+t')}\cvl{n\cv+\infty} 0.
\end{equation*}
Hence, since $\sousz$ is finite, for all $t'\in[0, \dt]$,
$\sum_{p\in\sousz} c_p J_p(\mu \ray(t_{k_n}))e^{-ip \angleS(t_{k_n})} e^{-ip t'}$ 
tends towards $c_0$ if $0\in I$, or towards $0$ otherwise,
as $n$ goes to $+\infty$.
Denote by $j_0$ the first zero of $J_0$.
Then $J_p(\ray)\neq0$ for any $\ray\in (-j_0, j_0)\setminus \{0\}$ and any $p\in\Z$.
Since for some $p\in\sousz$, $c_p\neq0$, we have  $J_p(\mu r(t_{k_n}))\to J_p(0)$, hence $r^\star=0$ since $|x|<\frac{\jj}{\mu}$ with $\jj<j_1<j_0$.
\end{proof}

\begin{rmrk}
In the above definition of $\zeta$, if there exist $k_1, k_2\in\Z$
with $|k_1|\neq |k_2|$, $c_{k_1}\neq0$ and $c_{k_2}\neq0$,
then $j_0=+\infty$ is a suitable choice due to the Bourget's hypothesis, proved by Siegel in \cite{siegel2014einige}.

\end{rmrk}

One can easily check that condition~\eqref{plong:eqCCA} is satisfied if $\mC = \langle\cdot, \zeta \rangle_Z$ for some $\zeta\in\doma$ (because $\mC^*\mC\mA = \langle\cdot, \mA^*\zeta\rangle\zeta$).
Thus, if $\mC$ has the form $\langle\cdot,\zeta\rangle_Z$ with $\zeta$ as in Proposition~\ref{prop:detec}, 
solutions of \eqref{eps} are such that $\mC\eps(t)\to0$ as $t\to+\infty$. In order to obtain the convergence of $\eps$ towards $0$, we need an additional observability assumption.
Let us recall the usual definition of approximate observability of \eqref{E:system_plonge_infinie} (see, \eg \cite{TW2009}).

\begin{dfntn}[Approximate observability]\label{def:app_obs}

System~\eqref{E:system_plonge_infinie} is said to be \emph{approximately observable} in some time $T>0$ for some
constant input
$u$
if and only if
\begin{equation}
    (\forall t\in[0, T],\ \opc\sg_t(\etat_0, u) = 0)
    \Longrightarrow
    \etat_0=0.
\end{equation}
\end{dfntn}
Since \eqref{E:system_plonge_infinie} is a linear system, Definition~\ref{def:app_obs} coincides with Definition~\ref{def:obs} in the finite-dimensional context.
\begin{assumption}[Isolated observability singularity]\label{ass:inobs}
Let $u\in[-\bornu, \bornu]$ where $\bornu = \bornuval$.
If $u\neq0$,
then the constant input $u$
makes \eqref{E:system_plonge_infinie}
approximately observable
in any positive time.

\end{assumption}

In \cite[Example 1]{Celle-etal.1989}, the authors investigated the observability of \eqref{E:system_plonge_infinie} in the case where
$\mu=1$ and
$\opc = \psX{\cdot}{\one}$ (\ie $\hh\circ h (x) = J_0(|x|)$, see Example~\ref{ex:Jk}).
Using a similar method, we prove the following.
\begin{prpstn}[Outputs satisfying Assumption~\ref{ass:inobs}]\label{prop:obs}
If one of the components of $\opc$ has the form $\psX{\cdot}{\zeta}$
where $\zeta = \sum_{k\in\sousz}c_ke_k\in\XX\setminus\{0\}$, with $\sousz\subset\Z$   finite,
and $\mu \bornu<j_0$ for some $j_0>0$, then \eqref{E:system_depart} satisfies Assumption~\ref{ass:inobs}.
\end{prpstn}

\begin{proof}
Let $\etat_0\in\XX$, $u\in\R\setminus\{0\}$ and $\etat(t) = \sg_t(\etat_0, u)$ be the unique corresponding solution of \eqref{E:system_plonge_infinie}.
The method of characteristics yields
\begin{align}
\psX{\etat(t)}{\zeta}
&=
\frac{1}{2\pi}\int_0^{2\pi}
e^{-i\mu u \int_0^t\sin(\varS-\sigma)\diff \sigma}\etat_0(\varS-t)
\sum_{k\in\sousz}\bar{c}_ke^{-ik\varS}
\diff \varS
\nonumber
\\
&=
\frac{1}{2\pi}\int_0^{2\pi}
\left(
\sum_{k\in\sousz}\bar{c}_k
e^{-i\mu u \cos(\varS)-ik\varS}
\right)
e^{i\mu u \cos(\varS-t)}\etat_0(\varS-t)
\diff \varS
\nonumber
\\
&=
\left(\psi*\uppsi_0\right)(t)
\nonumber
\end{align}
where $*$ denotes the convolution product over $\XX$,
$\psi:\varS\mapsto \sum_{k\in\sousz}\bar{c}_k
e^{-i\mu u \cos(\varS)-ik\varS}$
and
$\uppsi_0:\varS\mapsto e^{i\mu u \cos(\varS)}\etat_0(\varS)$.
Hence, according to Parseval's theorem,
\begin{align*}
\frac{1}{2\pi}
\int_0^{2\pi}|\psX{\etat(t)}{\zeta}|^2\diff t
=\norm{\psi*\uppsi_0}^2
=\normA{\hat{\psi}\cdot\hat{\uppsi}_0}{\hat{\XX}}^2
=\sum_{\ell\in\Z} |\psX{\psi}{e_\ell}|^2 |\psX{\uppsi_0}{e_\ell}|^2,
\end{align*}
where $\hat{\psi}$ (resp. $\hat{\uppsi}_0$) denotes the Fourier series coefficients of $\psi$ (resp $\uppsi_0$) in $\XX = L^2(\S^1, \C)\subset L^1(\S^1, \C)$ and $\hat\XX = l^2(\Z, \C)$.
Hence, it is sufficient to show that there exists $j_0>0$ such that, if $\mu u<j_0$, then $\psX{\psi}{e_\ell}\neq0$ for all $\ell\in\Z$. Indeed, it yields that if $\opc\etat(t) = 0$ for all $t\in[0, 2\pi]$, then $\uppsi_0 = 0$, \ie$\etat_0=0$, and
thus $u$ makes \eqref{E:system_plonge_infinie}
approximately
observable in time $2\pi$.

Note that
\begin{align}
\psX{\psi}{e_\ell}
=
\frac{1}{2\pi}\int_0^{2\pi}
\sum_{k\in\sousz}\bar{c}_k
e^{-i\mu u \cos(\varS)-i(k+\ell)\varS}
\diff \varS
=
\sum_{k\in\sousz}\bar{c}_k i^k
J_{k+\ell}(\mu u).
\tag{by \eqref{eq_Bessel}}
\end{align}
Set $d_k = \bar{c}_k i^k$
and
$F_\ell(r) = \sum_{k\in\sousz}d_k J_{k+\ell}(r)$ for all $r\in\R$.
Since $F_\ell$ is analytic for each $\ell\in\Z$,
its zeros are isolated.
Hence, for all $L>0$, there exists $j_0>0$ such that,
if $|\ell|<L$, then $F_\ell(r) \neq 0$ for all $r\in(-j_0, j_0)\setminus\{0\}$.
Now, let $\kmin = \min\{k\in\sousz\mid d_k\neq0\}$
and let us prove that there exists $j_0>0$ such that
$F_\ell(r) \neq 0$ for all $r\in(-j_0, j_0)\setminus\{0\}$ and all $\ell\geq-\kmin$.
(One can reason similarly for $\ell\leq \max\{k\in\sousz\mid d_k\neq0\}$).
We have
$F_\ell(r) = d_{\kmin}J_{\kmin+\ell}(r)\left(1 + \sum_{k\in\sousz}
\frac{d_k}{d_{\kmin}}\frac{J_{k+\ell}(r)}{J_{\kmin+\ell}(r)}\right)$.
According to \cite{neuman2004inequalities},
$\left|J_{k+\ell}(r)\right| \leq
\frac{1}{(k+\ell)!}\left(\frac{|r|}{2}\right)^{k+\ell}$
for all $r\in\R$.
Moreover, according to \cite{laforgia1986inequalities}, if $|r|\leq 1$, then
$$\left|J_{\kmin+\ell}(r)\right|\geq |r|^{\kmin+\ell}J_{\kmin+\ell}(1)
\geq \frac{|r|^{\kmin+\ell}}{(\kmin+\ell)!2^{\kmin+\ell}}\left(1 - \frac{1}{2(\kmin+\ell+1)}\right).
$$
Hence
\begin{align*}
|F_\ell(r)|
&\geq \left|d_{\kmin}\right|\left|J_{\kmin+\ell}(r)\right|
\left(1 - \sum_{k\in\sousz}
\frac{\left|d_k\right|}{\left|d_{\kmin}\right|}\frac{\left|J_{k+\ell}(r)\right|}{\left|J_{\kmin+\ell}(r)\right|}\right)\\
&\geq
\left|d_{\kmin}\right|\left|J_{\kmin+\ell}(r)\right|
\left(1 - 2\sum_{k\in\sousz}
\frac{\left|d_k\right|}{\left|d_{\kmin}\right|}
\left(\frac{|r|}{2}\right)^{k-\kmin}
\right).
\end{align*}
Hence, there exits $j_0>0$ such that, if 
$0<|r|<j_0$,
$|F_\ell(r)|\geq \frac{\left|d_{\kmin}\right|}{2} \left|J_{\kmin+\ell}(r)\right|$
for all $\ell\in\Z$.
Choosing $j_0\leq\min\{r>0\mid J_0(r)=0\}$,
one has $J_{\kmin+\ell}(r)\neq0$ for all $\ell\in\Z$,
hence $F_\ell(r)\neq0$.

In particular, if $\zeta = e_k$ for some $k\in\sousz$,
then $j_0 = \min\{r>0\mid J_0(r)=0\}$ is a suitable choice. Indeed, $J_k(\ray)\neq0$ for all $\ray\in (-j_0, j_0)\setminus\{0\}$ and all $k\in\Z$.
Hence, if $\mu \bornu < j_0$,
then $u$ makes \eqref{E:system_plonge_infinie}
approximately
observable in time $2\pi$.
Let $\upzeta(t) = \sg_t^*(\zeta, u)$, \ie
the solution of $\dot\upzeta = \mA(u(t))^*\upzeta$, $\upzeta(0)=\zeta$.
Since $\zeta$ is analytic, $t\mapsto\upzeta(t)$ is analytic as the unique solution of an analytic system.
Hence
$t\mapsto\psX{\etat(t)}{\zeta} = \psX{\etat_0}{\zeta(t)}$
is also analytic.
Thus $u$ makes \eqref{E:system_plonge_infinie}
approximately
observable in any positive time,
because $t\mapsto\psX{\etat(t)}{\zeta}$ vanishes on $[0, T]$ for some $T>0$ if and only if it vanishes on $[0, 2\pi]$.
\end{proof}

\begin{rmrk}
It is always possible to make $\mu\bornu<j_0$ by choosing $\kappa\jj$ and $\mu\delta$ small enough.
As explained in Remark~\ref{rem:gr}, the considered set of such maps $\opc$ is sufficient to approximate any output map $h$.
Moreover, if
$\psX{\cdot}{e_{k_1}}$ and $\psX{\cdot}{e_{k_2}}$
are two of linear forms of $\opc$
with $|k_1|\neq |k_2|$,
then $j_0=+\infty$ is a suitable choice 
due to the Bourget's hypothesis, proved by Siegel in \cite{siegel2014einige}.
\end{rmrk}

We are now in position to state the main result of Section~\ref{sec:example_infinie}.

\begin{thrm}\label{th:infinite}
Let $\Kx\subset\R^2$ be a compact set.
Let $R_0>0$ be such that $\Kx\subset B_{\R^2}(0, R_0)$.
Let 
$\mu_0, \delta_0, \dt_0>0$
be as in Corollary~\ref{cor:semiglob}.
Suppose there exists $\mu\in(0, \mu_0)$
such that for all $\delta\in(0, \delta_0)$ and all $\dt\in(0, \dt_0)$,
Assumptions~\ref{ass:h},~\ref{ass:detec} and~\ref{ass:inobs} are satisfied.
Then system \eqref{E:system_depart} is
stabilizable over
$\Kx$
by means of an infinite-dimensional piecewise constant dynamic output feedback.
Moreover, the closed-loop system is explicitly given by \eqref{E:system_stab_infinie}
for any $\alpha>0$ and
with
$\plong$ as in \eqref{E:deftau} and
$\inv$ as in \eqref{E:definv}.
\end{thrm}

Combining Propositions~\ref{prop:Jk},~\ref{prop:detec} and~\ref{prop:obs}, we obtain Theorem~\ref{cor:main} as an immediate corollary.
Regarding Corollary~\ref{cor:x2inf}, the proof is as follows.
\begin{proof}[Proof of Corollary~\ref{cor:x2inf}]
Let $\Kx\subset\R^2$ be a compact set.
Let $R_0>0$ be such that $\Kx\subset B_{\R^2}(0, R_0)$.
Let $\mu_0, \delta_0, \dt_0>0$ be as in Corollary~\ref{cor:semiglob}.
According to Example~\ref{ex:Jk} and Proposition~\ref{prop:detec}, Assumptions~\ref{ass:h} and~\ref{ass:detec} are satisfied for any $\mu\in(0,\mu_0)$ by considering 
$\hh:y\mapsto J_0(\mu\sqrt{2y})$.
Moreover, by choosing $\kappa\jj<j_0$ and $\delta<\frac{j_0-\kappa\jj}{16\nu^2\mu}$,
Assumption~\ref{ass:inobs} is also satisfied according to Proposition~\ref{prop:obs}.
Hence, Theorem~\ref{th:infinite} does apply on $\Kx$.
\end{proof}

\subsection{Proof of Theorem~\ref{th:infinite}}
Let $\Kx$ be a compact subset of $\R^2$.
Let $\Ra>0$ be such that $\Kx\subset\B_{\R^2}(0, \Ra)$
and $\mu_0, \delta_0, \dt_0>0$
be as in Corollary~\ref{cor:semiglob}.
This implies the statement \ref{defi} of Definition~\ref{def:stab_inf}, 
with $\Kz=\plong(\Kx)$.
Only \ref{defii} and \ref{defiii} remain to be proved.
Let $\alpha>0$,
$\plong$ be as in \eqref{E:deftau} and
$\inv$ be as in \eqref{E:definv}.
Let $x_0$ and $\xhat_0$ be in $\Kx$,
$(x, \etath)$ be the corresponding solution of \eqref{E:system_plonge_infinie},
$\etat=\plong(x)$, $\eps = \etath-\etat$ and $u = \phi(\inv(\etath)) + \delta\psn^2(\etath-\one)$.
Remark that
\begin{align*}
    \norm{\eps} - \Lplong|x|
    &\leq \norm{\etath-\one} + \norm{\etat-\one} - \Lplong|x|\\
    &=\norm{\etath-\one} + \norm{\plong(x)-\plong(0)} - \Lplong|x|\\
    &\leq \norm{\etath-\one}
\end{align*}
and
\begin{align*}
    \norm{\etath-\one}
    \leq \norm{\eps} + \norm{\etat-\one}
    \leq \norm{\eps} + \Lplong|x|
\end{align*}
where $\Lplong$ is the Lipschitz constant of $\plong$ over $\Kx$.
Hence proving statement \ref{defii} of Definition~\ref{def:stab_inf} reduces to prove
(again with $\Kz=\plong(\Kx)$)
that for some $\mu, \delta$ and $\dt$ small enough,
\begin{enumerate}[label = \textit{(ii')}]
    \item
    \label{defii2}
    For all $R_x, R_{\eps}>0$, there exists $r_x, r_{\eps}>0$ such that
    for all $(x_0, \etath_0)\in\Kx\times\plong(\Kx)$,
    if $|x_0|<r_x$ and $\norm{\eps_0}<r_{\eps}$,
    then
    $|x(t)|<R_x$ and $\norm{\eps(t)}<R_{\eps}$ for all $t\geq0$.
\end{enumerate}
Since $\plong$ is continuous, if $x\to0$, then $\plong(x)\to\one$, and, a fortiori, $\plong(x)\cvf\one$.
Hence proving statement \ref{defiii} of Definition~\ref{def:stab_inf}
reduces to prove
that for some $\mu, \delta$ and $\dt$ small enough,
\begin{enumerate}[label = \textit{(iii')}]
    \item
    \label{defiii2}
    $x(t)\to0$ and $\eps(t)\cvf0$ as $t$ goes to infinity.
\end{enumerate}
We prove \ref{defii2} in Section~\ref{sec:stability_infinie} and \ref{defiii2} in Section~\ref{sec:attractivity_infinie}.

\subsubsection{Stability}\label{sec:stability_infinie}

In order to prove stability, we reason as in Section~\ref{sec:stability}.
Let $R_x, R_{\eps}>0$. We seek $r_x, r_{\eps}>0$ such that
for all $(x_0, \etath_0)\in\Kx\times\plong(\Kw)$,
if $|x_0|<r_x$ and $\norm{\eps_0}<r_{\eps}$,
then
$|x(t)|<R_x$ and $\norm{\eps(t)}<R_{\eps}$ for all $t\geq0$.
Since $\norm{\eps}$ is non-increasing,
choosing $r_\eps \leq R_\eps$ proves the stability of $\eps$.
The dynamics of $x(t)$ can be written as:
\begin{align}
    \dot x(t) &=
    (A+bK)x(t)
    + \delta \psn^2(\plong(x(t))-\one) b\label{ligne1}\\
    &\quad + bK(x(t_k) - x(t))
    +\delta\psn^2(\plong(x(t_k))-\one)b
    -\delta \psn^2(\plong(x(t))-\one)b\label{ligne2}\\
    &\quad +
    bK(\inv(\etath(t_k)-x(t_k))
    +\delta\psn^2(\etath(t_k)-\one)b
    -\delta \psn^2(\plong(x(t_k))-\one)b\label{ligne3}
\end{align}
First, we show that \eqref{ligne1} is a locally exponentially stable dynamical system when $\delta$ is small enough.
Indeed,
$A+bK$ is Hurwitz
and for all $\xi_1, \xi_2$ in $Z$,
\begin{align}
    \left|\psn^2(\xi_1) - \psn^2(\xi_2)\right|
    &\leq\sum_{k\in\N} \frac{1}{k^2+1}
    \left||\psX{\xi_1}{e_k}|^2 - |\psX{\xi_2}{e_k}|^2\right|\nonumber\\
    &\leq\sum_{k\in\N} \frac{1}{k^2+1}
    |\psX{\xi_1-\xi_2}{e_k}|\left(|\psX{\xi_1}{e_k}|+|\psX{\xi_2}{e_k}|\right)\nonumber\\
    &\leq \nu^2
    \|\xi_1-\xi_2\|_Z(\|\xi_1\|_Z+\|\xi_2\|_Z).\label{Nlip}
\end{align}
Hence, for all $x_1, x_2$ in $\R^2$,
\begin{align*}
    \left|\delta \psn^2(\plong(x_1)-\one)
    - \delta \psn^2(\plong(x_2)-\one)\right|
    \leq 4\delta \nu^2\Lplong|x_1-x_2|.
\end{align*}
Let $P\in\R^{2\times 2}$ be positive definite such that 
$P(A+bK) + (A+bK)'P < -2 \Id_{\R^2}$.
If $\chi$ is a solution of \eqref{ligne1},
then we have
\begin{align*}
\frac{\diff}{\diff t}
(\chi'P\chi)(t)
\leq -2|\chi(t)|^2 + 4\delta \nu^2\Lplong|Pb||\chi(t)|^2
\end{align*}
Thus, by choosing $\delta<\min(\delta_0, \frac{1}{4\nu^2\Lplong|Pb|})$, we get
 $\frac{\diff}{\diff t}
(\chi'P\chi)(t)
\leq -|\chi(t)|^2
$,
hence the local exponential stability of
\eqref{ligne1}.
Now, we show that the perturbation term \eqref{ligne2} preserve stability.
If $\chi$ is a solution of \eqref{ligne1}-\eqref{ligne2},
then, with the variation of constants,
\begin{align*}
    \chi(t)
    = e^{(t-t_k)(A+bK)}\chi(t_k)
    +\int_{t_k}^t e^{(t-s)(A+bK)}
    b\left(K(\chi(t_k)-\chi(s))+
    \delta\psn^2(\plong(\chi(t_k))-\one)\right)\diff s.
\end{align*}
Hence, for all $k\in\N$ and all $t\in[t_k, t_{k+1})$,
\begin{align*}
    |\chi(t)-\chi(t_k)|
    &\leq \left\|\Id_{\R^2} - e^{(t_k-t)(A+bK)} \right\| |\chi(t)|\\
    &\quad+
    \int_{t_k}^t \left\|e^{(t_k-s)(A+bK)}\right\|
    \left(
    \kappa |\chi(t_k)-\chi(s)|
    +
    4\delta \nu^2\Lplong|\chi(t_k)|\right)
    \diff s\\
    &\leq\left(
    \left\|\Id_{\R^2} - e^{(t_k-t)(A+bK)} \right\|
    +
    4\delta \nu^2\Lplong
    \int_{t_k}^t \left\|e^{(t_k-s)(A+bK)}\right\|
    \diff s\right)
    |\chi(t)|\\
    &\quad+
    \int_{t_k}^t \left\|e^{(t_k-s)(A+bK)}\right\|
    \kappa |\chi(t_k)-\chi(s)|\diff s\\
    &\quad+
    \int_{t_k}^t \left\|e^{(t_k-s)(A+bK)}\right\|
    4\delta \nu^2\Lplong
    |\chi(t)-\chi(t_k)|
    \diff s.\\
\end{align*}
Since $\left\|\Id_{\R^2}-e^{(t_k-t)(A+bK)}\right\|\leq \dt\|A+bK\|e^{\dt\|A+bK\|}$
and $\int_{t_k}^{t} \left\|e^{(t_k-s)(A+bK)}\right\|\diff s\leq M$ for some $M>0$ independent of $t$ and $k$ (as in the proof of \eqref{E:varcons}), we obtain 
\begin{multline*}
    (1 - 4\delta \nu^2\Lplong M)|\chi(t)-\chi(t_k)|
    \leq(
    \dt\|A+bK\|e^{\dt\|A+bK\|}
    +
    4\delta \nu^2\Lplong
    M)
    |\chi(t)|\\
    +
    \int_{t_k}^t \left\|e^{(t_k-s)(A+bK)}\right\|
    \kappa |\chi(t_k)-\chi(s)|\diff s.
\end{multline*}
Choosing $\delta<\frac{1}{8\nu^2\Lplong M}$, we obtain
by Grönwall's lemma,
\begin{align}\label{gronwall}
    |\chi(t)-\chi(t_k)|\leq 2\left(\dt\|A+bK\|e^{\dt\|A+bK\|}+4\delta \nu^2\Lplong  M\right)e^{2\dt\kappa M}|\chi(t)|
\end{align}
Using the fact that with $\delta$ as above
$\frac{\diff}{\diff t}
(\chi'P\chi)(t)
\leq -|\chi(t)|^2
$ if $\chi$ satisfies \eqref{ligne1}, we also get if $\chi$ satisfies \eqref{ligne1}-\eqref{ligne2} that
\begin{align}\label{lyapgronwall}
\frac{\diff}{\diff t}
(\chi'P\chi)(t)
\leq -|\chi(t)|^2 + 
|\chi(t)| |Pb|
\left(\kappa |\chi(t_k)-\chi(t)|
+4\delta \nu^2\Lplong|\chi(t_k)-\chi(t)|
\right)
\end{align}
Combining \eqref{gronwall} and \eqref{lyapgronwall} shows that \eqref{ligne1}-\eqref{ligne2} is still locally exponentially stable by choosing $\delta$ and $\dt$ small enough.
Finally, we show that the perturbation term \eqref{ligne3} preserves stability.
Since the dynamical system \eqref{ligne1}-\eqref{ligne2} is locally exponentially stable, there exists $r_x>0$ and $\eta>0$ such that, if $|x_0|\leq r_x$ and the perturbation term \eqref{ligne3} is bounded by $\eta$,
then $|x(t)|\leq R_x$ for all $t\in\R_+$.
Since we have
\begin{align*}
    |bK\inv(\etath(t_k))-x(t_k)|
    \leq \kappa\Linv \norm{\eps(t_k)} 
    \leq \kappa\Linv r_\eps
\end{align*}
and (by \eqref{Nlip}),
\begin{align*}
    \left|\delta\psn^2(\etath(t_k)-\one)
    -\delta \psn^2(\plong(x(t_k))-\one)
    \right|
    \leq
    \delta\nu^2\|\eps(t_k)\|_Z(\|\zhat(t_k)\|_Z+\|\tau(x(t_k))\|_Z)
    \leq \delta\nu^2r_\eps(r_\eps+2)
\end{align*}
we obtain the desired result by choosing
\begin{equation}
    \kappa\Linv r_\eps + \delta\nu^2r_\eps(r_\eps+2) \leq \eta.
\end{equation}

\subsubsection{Attractivity}\label{sec:attractivity_infinie}
\noindent
\textbf{Step 1: Show that $\boldsymbol{\eps\cvf 0}$.}
Let $\Omega$ be the set
of weak limit points of $(\epsilon(t))_{t\in\R_+}$.
According to \eqref{E:eps_decroit_infinie}, $\eps$ is bounded.
Hence, by Alaoglu's theorem, $\Omega$ is not empty.
It remains to show that $\Omega = \{0\}$.
Let $\eps^\star\in\Omega$ and an increasing sequence $(t_n)_{n\in\N}
\cv+\infty
$
such that $\eps(t_{n})\cvf\eps^\star$ as $n\cv+\infty$.

For all $n\in\N$, let $k_n$ be such that $t_n\in[t_{k_n}, t_{k_n+1})$.
Again, after passing to a subsequence,
we may assume that $(t_{k_n})$ is increasing.
Let $p\in\N$ to be fixed later.
Consider the sequences
$(u(t_{k_n})), (u(t_{k_n-1})),\dots,(u(t_{k_n-p}))$ (with $t_\ell=0$ if $\ell<0$).
According to Corollary~\ref{cor:semiglob}, $|u|$ is bounded by $\bornuval$.
Moreover, combining 
\eqref{E:etat_constant} and \eqref{E:eps_decroit_infinie},
$\etath$ is also bounded.
Hence, after passing to a subsequence, we may assume that $(u(t_{k_{n}})), (u(t_{k_{n}-1})),\dots,(u(t_{k_{n}-p}))$ converge and $\etath(t_{k_n-p})$ converges weakly to some $\etath^\star\in \XX$.

If there exists $j\in\{0,\dots,p\}$ such that $(u(t_{k_{n}-j}))_{n\in\N}$ does not converge to $0$, then the arguments used with persistance assumptions (such as the ones found in \cite{Celle-etal.1989}) remain valid.
According to Assumption~\ref{ass:inobs}, $(u(t_{k_{n}-j}))_{n\in\N}$ has an accumulation point $u^\star$ that makes the system~\eqref{E:system_plonge_infinie} approximately observable in any positive time.
Since $u(t) = u(t_{k_{n}-j})$ for all $t\in[t_{k_{n}-j}, t_{k_{n}-j+1})$
we obtain by \cite[Theorem 3.5]{brivadis:hal-02529820} (see also \cite[Theorem 7, Step 4]{Celle-etal.1989}) that $\eps^\star=0$.

Now, assume that $u(t_{k_{n}-j})\to0$ as $n\to+\infty$ for all $j\in\{0,\dots,p\}$.
The system is sufficiently explicit to allow computation of  $\etath(t_{k_{n}-j})$  from the knowledge of $\etath(t_{k_{n}-p})$ for all $j\in\{0,\dots,p-1\}$.
The assumption $u(t_{k_{n}-j})\to0$ as $n\to+\infty$ then implies strong constraints on the weak limit $\etath^\star$ of $\etath(t_{k_n}-p)$. With $p=2$, we actually obtain that the only possible case corresponds to $(\etath(t_n),\varepsilon(t_n))\tow(\one,0)$. This takes the remainder of the step.

Passing to the limit in the expression of $u(t_{k_n-p})$, we get the existence of $\psnlim\in\R_+$ such that $\psn^2(\etath(t_{k_n-p})-\one) \to \psnlim$ and 
\begin{align}
    K \invf(\psX{\etath^\star}{e_1}) + \delta \psnlim = 0.
\end{align}
Using the method of characteristics, one can show that for all $t, \ttau\in\R_+$ and almost all $\varS\in\S^1$,
\begin{align}\label{E:uconst}
    \etat(t+\ttau, \varS) = \carac(t+\ttau, t, \varS)\etat(t, \varS-\ttau).
\end{align}
where $\carac(t+\ttau, t, \varS) = e^{-i\mu \int_{t}^{t+\ttau}u(\sigma)\sin(\varS-\sigma)\diff \sigma}$.
Then, according to Duhamel's formula,
\begin{align}\label{E:xihat}
    \etath(t+\ttau, \varS) = \carac(t+\ttau, t, \varS)\etath(t, \varS-\ttau) -
    \alpha
    \int_{t}^{t+\ttau} \carac(t+\ttau, \sigma, \varS)
    \big(\left(\opc^*\opc\eps(\sigma)\right)(\varS-\ttau)\big)\diff \sigma.
\end{align}
Since
$u(t) = u(t_{k_{n}-j})$ for all $t\in[t_{k_{n}-j}, t_{k_{n}-j+1})$
and
$u(t_{k_{n}-j})\to0$ as $n\to+\infty$ for all $j\in\{0,\dots,p\}$,
we have that
$\carac(t_{k_{n}-p}+t', t_{k_{n}-p}, \varS)\to1$ as $n\to+\infty$, uniformly in $\varS\in\S^1$, for all $t'\in[0, p\dt]$.
Then
\begin{multline}\label{E:ineqzhat}
    \norm{\etath(t_{k_{n}-p}+t', \cdot) - \etath(t_{k_{n}-p}, \cdot-t')}
    \leq \sup_{\varS\in\S^1}|\carac(t_{k_{n}-p}+t', t_{k_{n}-p}, \varS) - 1| \norm{\etath(t_{k_{n}-p})} \\
    + \alpha \sup_{\substack{\sigma\in[t_{k_{n}-p}, t_{k_{n}-p}+t']\\ \varS\in\S^1}}|\carac(t_{k_{n}-p}+t', \sigma, \varS)|
    \norm{\int_{t_{k_{n}-p}}^{t_{k_{n}-p}+t'}\opc^*\opc\eps(\sigma)\diff \sigma}
\end{multline}
tends towards $0$ as $n$ goes to $+\infty$
since $\etath$ and $\eps$ are bounded,
$t\mapsto\norm{\mC\varepsilon(t)}$ is integrable over $\R_+$
(see \eqref{E:eps_decroit_infinie})
Hence, for all $t'\in[0, p\dt]$,
\begin{align}\label{E:cvzhat1}
    \psX{\etath(t_{k_{n}-p}+t',\cdot)}{e_1}
    \to \psX{\etath^\star(\cdot-t')}{e_1}
    = e^{-it'} \psX{\etath^\star}{e_1}
\end{align}
and
\begin{align}\label{E:cvzhat2}
   \psn^2(\etath(t_{k_n-p}+t', \cdot)-\one)
   \to \psn^2(\etath^\star(\cdot-t')-\one)
   = \psnlim.
\end{align}
as $n$ goes to $+\infty$.
According to \eqref{E:cvzhat1} and \eqref{E:cvzhat2}
for $t'=j\dt$ and $j\in\{0,\dots,p\}$,
and since $u(t_{k_{n}-j})\to0$, we get
\begin{equation}\label{syst:absurde}
    K \invf(e^{-ij\dt} \psX{\etath^\star}{e_1}) + \delta \psnlim = 0,\qquad \forall j\in\{0,\dots,p\}.
\end{equation}
For all $t\in\R$ and all $\zeta\in \B_{\C}\left(0, J_1\left(\jj\right)\right)$, we have by \eqref{E:def_f},
$
K \invf(e^{i t} \zeta)
= K \frakr(t) \invf(\zeta)
$
where 
$\frakr(t) = \begin{pmatrix}
\cos t &\sin t\\
-\sin t&\cos t
\end{pmatrix}$.
Equations \eqref{syst:absurde}
with $p=2$ 
can be rewritten as the matrix equality
\begin{equation*}
\begin{pmatrix}
K & 1\\
K \frakr(\dt) & 1\\
K \frakr(2\dt) & 1\\
\end{pmatrix}
\begin{pmatrix}
K\invf(\psX{\etath^\star}{e_1})\\
\delta \psnlim
\end{pmatrix} = 0.
\end{equation*}
Since the square matrix on the left hand side is invertible for $\dt\in(0, \pi)$, we get that $\psnlim = 0$ \emph{i.e.} $\etath(t_{k_{n}-p}) \tow \one$.
Combining it with \eqref{E:ineqzhat}, we have
$\etath(t_{k_{n}}-t') \tow \one$ as $n$ goes to $+\infty$
for all $t'\in[0, p\dt]$.
In particular, $\opc \etath(t_{k_{n}}-t') \to \opc\plong(0)$.
Since $\opc\eps(t)\to 0$ as $t\to+\infty$ by \eqref{E:eps_decroit_infinie}, we obtain $\opc\plong(x(t_{k_{n}}-t')) \to \opc\plong(0)$ as $n\to+\infty$.
Hence, by Assumption~\ref{ass:detec}, $x(t_{k_{n}})\to0$ \emph{i.e.} $\etat(t_{k_{n}})\to\one$.
Thus $\eps(t_{k_{n}})\tow 0$ \emph{i.e.} $\eps^\star=0$.

\medskip
\noindent
\textbf{Step 2: Show that $\boldsymbol{x\cv0}$.}
Recall that $x$ satisfies the following dynamics:
\begin{align*}
    \xdot = (\matR+bK) x + bK(\inv(\etath)-x) + \delta\psn^2(\etath-\one) b.
\end{align*}
Since $A+bK$ is Hurwitz, there exists $P\in\R^{2\times 2}$ positive definite such that $P(A+bK) + (A+bK)'P < -2 \Id_{\R^2}$.
Set $V:\R^2\ni x\mapsto x'Px$.
Then
\begin{align*}
    \frac{\diff}{\diff t} V(x)
    &\leq -2 |x|^2 + 2|x| |Pb| \kappa |\inv(\etath)-x|
    + 2|x||Pb|\delta\psn^2(\etath-\one)\\
    &\leq -2 |x|^2 + 2|Pb|\frac{\jj}{\mu}\left(\kappa |\inv(\etath)-x|+
    \delta\psn^2(\etath-\one)\right).
\end{align*}
We have
\begin{align*}
    \psn(\etath-\one)
    \leq \psn(\eps) + \psn(\etat-\one)
    \leq \psn(\eps) + \nu\norm{\etat-\one}
    \leq \psn(\eps) + \nu \Lplong|x|
\end{align*}
where $\Lplong$ is the Lipschitz constant of $\plong$ over $\Kx$.
Hence, if $\delta\leq\frac{\mu}{4|Pb|j\nu^2\Lplong^2}$
(which we can assume without loss of generality by replacing $\delta_0$ by $\min(\delta_0, \frac{\mu}{4|Pb|j\nu^2\Lplong^2})$, since diminishing $\delta$ ),
then
\begin{align*}
    \frac{\diff}{\diff t} V(x)
    &\leq -|x|^2
    + 2|Pb|\frac{\jj}{\mu}\left(\kappa|\inv(\etath)-x|
    + 2\delta\psn^2(\eps)\right).
\end{align*}
Recall that $|x|$ and $|\inv(\etath)|$ are bounded by $\frac{\jj}{\mu}$.
Moreover, $\psn(\eps(t))\to0$ as $t\to+\infty$ by Step~1,
and $\inv(\etath)-x\to0$ as $t\to+\infty$ since $\inv$ is a strong left-inverse of $\plong$ (see Corollary~\ref{cor:conv_faible_forte}).

For all $r>0$, set $D(r)=\{x\in\R^2\mid V(x)\leq r\}$.
In order to prove that $x\to0$, we show that for all $r>0$, there exists $T(r)\geq0$ such that $x(t)\in D(r)$ for all $t\geq T(r)$.
If $r>0$ is such that $\bar{B}_{\R^2}(0, \frac{\jj}{\mu})\subset D(r)$ then $T(r)=0$ satisfies the statement.
Let $0<r<R$ be such that $\bar{B}_{\R^2}(0, \frac{\jj}{\mu})\not \subset  D(r)$ and $\bar{B}_{\R^2}(0, \frac{\jj}{\mu})\subset D(R)$.
Since $\psn(\eps(t))\to 0$ and $\inv(\etath(t))-x(t)\to 0$, there exists $T_1(r)>0$ such that for all $t\geq T_1(r)$, if $x(t)\not\in D(r)$, then
$
\frac{\diff}{\diff t}V(x)<-\bar{m},
$
for some $\bar{m}>0$.
First, this implies that if $x(t)\in D(r)$ for some $t\geq T_1(r)$, then $x(s)\in D(r)$ for all $s\geq t$.
Second, for all $t\geq0$,
\begin{align*}
    V(x(T_1(r)+t))
    &=
    V(x(T_1(r)))
    + \int_0^{t}\frac{\diff}{\diff \tau}V(x(T_1(r)+\tau))\diff \tau
    \leq R - \bar{m}t 
    \tag{while $x(T_1(r)+t)\notin D(r)$.}
\end{align*}
Set $T_2(r) = \frac{R-r}{m}$ and $T(r) = T_1(r) + T_2(r)$. Then for all $t\geq T(r)$, $x(t)\in D(r)$, which concludes the proof.

\section{Conclusion}

The goal of the paper was to illustrate new approaches to tackle the problem of output stabilization at an unobservable target. As we explained, Luenberger observers can be employed, as long as some embedding is provided. To mitigate the effect of the observability singularity, we rely on the dissipativity of the error system and perturbations of the feedback law. The promise of such strategies is well exemplified by Theorem~\ref{th:finite}, where we were able to set up this approach. However, while classical output linearization methods allow to introduce linear observers, the dissipativity is not guaranteed. 
Relying on representation theory, we explored, for a specific control system, a different approach to obtain an embedding of the dynamics into a unitary infinite dimensional system. This allowed to design an observer for many nonlinear outputs of the original system while guaranteeing the dissipative nature of the error system, with Theorem~\ref{cor:main} also covering some of the cases treated in
Theorem~\ref{th:finite}.

Beyond the method we explored in the present article, we wish to stress that topological obstructions to output feedback stabilization
can be lifted when infinite-dimensional observers are considered.
More precisely, the obstruction brought up in \cite{Coron1994} regarding the stabilizability of $\dot x = u,\, y = x^2$ and extended in Corollary~\ref{cor:impossible} vanishes if one extends the usual definition of 
dynamic output feedback stabilizability
by allowing infinite-dimensional states fed by the output, as in Definition~\ref{def:stab_inf}.

Therefore, new infinite-dimensional embedding techniques for output feedback stabilization,
either based on the more general framework of \cite{Celle-etal.1989},
or on other infinite-dimensional observers, need to be investigated.

\subsubsection*{Acknowledgements}
The authors would like to thank Vincent Andrieu
for many fruitful discussions.\\
This research was partially funded by the French Grant ANR ODISSE (ANR-19-CE48-0004-01)
and by the ANR SRGI (ANR-15-CE40-0018).

\appendix

\section{The matrix \texorpdfstring{$\mQ$}{} is invertible}\label{app:det}
Let us compute the determinant of $\mQ$.

\begin{align*}
\det \mQ
=\begin{vmatrix}
K & \delta & 0\\
K \rot & 0 & - \delta \alpha\\
\vdots & \vdots & \vdots \\
K \rot^{n+1} & 0 & \delta(-\alpha)^{n+1}
\end{vmatrix}
&=(-1)^{n+1}\delta^2\alpha
\begin{vmatrix}
K \rot && 1\\
\vdots && \vdots \\
K \rot^{n+1} && (-\alpha)^n
\end{vmatrix}=-\delta^2\alpha
\sum_{k=0}^n \alpha^k Q(k)
\end{align*}
where
\[
Q(k)
=\begin{vmatrix}
\tilde K\rot^{0} \\ \vdots \\ \tilde K\rot^{k-1} \\ \tilde K\rot^{k+1} \\ \vdots \\ \tilde K \rot^{n}
\end{vmatrix},\qquad
\tilde K = K\rot,\qquad k\in\{0, \dots, n\}.
\]
Let $P(X)=\sum_{k=0}^n c_k X^k$ be the characteristic polynomial of $\rot$. Since $\rot$ is skew-symmetric and invertible, it holds that $n$ is even, $P$ is minimal for $\rot$, positive on $\R$, $c_n=1$. Then,
$$
A^n=-\sum_{k=0}^{n-1} c_k \rot^k. 
$$
Let $\Delta$ be the determinant of the Kalman observability matrix of $(\tilde{K}, \rot)$.
Since $(K, \rot)$ is observable and $\rot$ is invertible, $\Delta\neq0$.
Then for $k<n$,
$$
Q(k)
=
\begin{vmatrix}
\tilde K\rot^{0} \\ \vdots \\ \tilde K\rot^{k-1} \\ \tilde K\rot^{k+1} \\ \vdots 
\\ 
 \sum_{i=0}^{n-1} c_i \tilde K\rot^i.
\end{vmatrix}
=
\begin{vmatrix}
\tilde K\rot^{0} \\ \vdots \\ \tilde K\rot^{k-1} \\ \tilde K\rot^{k+1} \\ \vdots 
\\ 
- c_k \tilde K\rot^k.
\end{vmatrix}
=-c_k(-1)^{n-k} \Delta. 
$$
The case $k=n$ simply yields $Q(n)=\Delta$.
Then
$$
\det \mQ
=
\delta^2\alpha \Delta
\sum_{k=0}^n  c_k(-1)^{k}\alpha^k  
=
\delta^2\alpha \Delta P(-\alpha).
$$
Since $P$ is positive on $\R$, $\det \mQ\neq0$ as soon as $\alpha>0$.


\bibliographystyle{plain}
\bibliography{references}

\end{document}